\documentclass{siamltex}
\usepackage{amsmath,amsfonts,amssymb}
\usepackage{latexsym}
\usepackage{graphicx,overpic}
\usepackage{subfigure,caption}
\usepackage{pgfplots}
\pgfplotsset{compat=newest}
\usepackage{color}
\usepackage{booktabs}
\usepackage{caption}
\usepackage{lscape}
\usepackage{tikz}
\usepackage{bm}

\usepackage{tikz}

\newcommand{\bA}{\mathbf A}
\newcommand{\bB}{\mathbf B}

\newcommand{\bH}{\mathbf H}
\newcommand{\bI}{\mathbf I}

\newcommand{\bP}{\mathbf P}

\newcommand{\bV}{\mathbf V}

\newcommand{\bn}{\mathbf n}

\newcommand{\bu}{\mathbf u}

\newcommand{\bv}{\mathbf v}
\newcommand{\bw}{\mathbf w}

\newcommand{\T}{\mathcal T}

\newcommand{\tr}{{\rm tr}}

\newcommand{\divG}{{\mathop{\,\rm div}}_{\Gamma}}
\newcommand{\gradG}{\nabla_{\Gamma}}

\newcommand{\gradGh}{\nabla_{\Gamma_h}}

\newcommand{\OGamma}{\Omega^\Gamma_h}

\newcommand{\cE}{\mathcal E}

\newcommand{\cT}{\mathcal T}

\newcommand{\tlam}{\tilde \lambda}

\newcommand{\CEl}{\cE_h^{-\ell}}
\newcommand{\kmax}{k_{\max}}
\newcommand{\He}{H^{\rm ex}}
\newcommand{\Paj}{P_{a_h,j}}
\newcommand{\Pbj}{P_{b_h,j}}
\def\Enorm#1{|\!\| #1 \|\!|_h}
\def\Ecnorm#1{|\!\| #1 \|\!|}

\newtheorem{assumption}{Assumption}[section]
\newtheorem{remark}{Remark}[section]

\numberwithin{equation}{section}
\begin{document}
\title{Analysis of finite element methods for surface vector-Laplace eigenproblems}
\author{
Arnold Reusken\thanks{Institut f\"ur Geometrie und Praktische  Mathematik, RWTH-Aachen
University, D-52056 Aachen, Germany (reusken@igpm.rwth-aachen.de).}
}
\maketitle

\begin{abstract} In this paper we study finite element discretizations of a surface vector-Laplace eigenproblem. We consider two known classes of finite element methods, namely one based on a vector analogon of the Dziuk-Elliott surface finite element method and one based on the so-called trace finite element technique. A key ingredient in both classes of methods is a penalization method that is used to enforce tangentiality of the vector field in a weak sense.  This penalization and the perturbations that arise from numerical approximation of the surface lead to essential \emph{nonconformities} in the discretization of the variational formulation of the vector-Laplace eigenproblem. We present a general abstract framework applicable to such nonconforming discretizations of eigenproblems. Error bounds both for eigenvalue and eigenvector approximations are derived that depend on certain consistency and approximability parameters. Sharpness of these bounds is discussed. Results of a numerical experiment illustrate certain convergence properties of such finite element discretizations of the surface vector-Laplace eigenproblem. 
\end{abstract}
\section{Introduction} In recent years there has been a strongly growing interest in the field of modeling and numerical simulation of surface fluids based on Navier-Stokes type PDEs on (evolving) surfaces  \cite{arroyo2009,Jankuhn1,Kobaetal_QAM_2017,miura2017singular,Nitschkeetal_arXiv_2018,reuther2015interplay}. 
Navier-Stokes equations posed on manifolds is  a classical topic in  analysis, cf., e.g., \cite{ebin1970groups,mitrea2001navier,taylor1992analysis,Temam88}. The development and (error) analysis of  numerical methods for surface (Navier-)Stokes equations has been studied in recent literature, e.g.,  \cite{nitschke2012finite,reuther2015interplay,reusken2018stream,reuther2018solving,fries2018higher,olshanskii2018finite,olshanskii2019penalty,Bonito2019a,OlshanskiiZhiliakov2019,Lederer2019}. 
 In all these papers  finite element discretization methods are treated. In almost all papers on finite element discretizations of surface (Navier-)Stokes equations the key condition that  the velocity has to be tangential to the surface is handled by a penalty technique. In such a method  nontangential components are allowed in the discretization but their magnitude is made sufficiently small by appropriate penalization. Alternatively, for surfaces that are simply connected, one can use an approach based  on a stream function formulation \cite{nitschke2012finite,reusken2018stream}. Another alternative that avoids penalization is introduced in the recent papers \cite{Lederer2019,Bonito2019a}, in which a surface finite element approach is combined with a Piola transformation for the construction of divergence-free tangential finite elements. 

Most of the above-mentioned  papers on finite element  methods treat the discretization of surface (Navier-)Stokes equations. In the papers  \cite{hansbo2016analysis,grossvectorlaplace,Jankuhn4} finite element discretizations of surface vector-Laplace equations are studied. 
In none of these papers, or in any other paper that we know of, the discretization of a vector-Laplace \emph{eigenproblem} has been studied. In the recent paper \cite{Bonito2019a} an error analysis of surface finite element discretizations for a \emph{scalar} Laplace-Beltrami eigenproblem is presented. In this paper we analyze finite element discretizations of a \emph{vector-}Laplace eigenproblem of the form
\begin{equation} \label{EV1}
  - \hat \Delta_\Gamma\bu + \bu = \lambda \bu \quad \text{on}~~\Gamma,
\end{equation}
where $\Gamma$ is a closed smooth two-dimensional surface. The eigenfunction $\bu$ is a field tangential to $\Gamma$. The vector-Laplace operator that we study is of the form $-\hat \Delta_\Gamma \bu :=-\tfrac12 \bP\divG(\gradG \bu  + (\gradG \bu)^T)$, but the analysis also applies to variants of this Laplacian. Precise definitions of $\bP$ and the tangential differential operators involved in $\hat \Delta_\Gamma$ are given in section~\ref{Sectcontinuous}. Clearly, such vector-Laplace eigenproblems are of interest in the field of surface (Navier-)Stokes equations. A further motivation for such eigenproblems comes from applications of so-called approximate Killing vector fields in computer graphics \cite{Azencot2015,BenChen2010,Solomon2011,Azencot2013,Tao2016}.

For discretization of the  eigenproblem \eqref{EV1} we restrict to the most popular class of finite element methods used for discretization of surface (Navier-)Stokes equations, namely those that combine standard finite element space used for scalar surface PDEs with a penalty approach, e.g. \cite{reuther2018solving,fries2018higher,olshanskii2018finite,olshanskii2019penalty}. 

There is extensive literature on the analysis of finite element discretizations of elliptic eigenproblems, cf. the overview papers \cite{Babuska,Boffi} and references therein. We also refer to the seminal paper~\cite{Knyazev} in which several approaches for the analysis of variational Galerkin methods for elliptic eigenproblems are discussed.   
The analyses presented in these papers apply to a \emph{conforming} Galerkin setting, in the sense that the eigenproblem discretization is determined in a subspace of the Hilbert space in which the original eigenproblem is posed. 
There are some papers in which this theory is adapted to a \emph{non}conforming setting. For example, in
 \cite{Antonietti} a class of discontinuous Galerkin finite element nonconforming discretizations for the scalar Laplace eigenproblem is analyzed. 
Also in the discrete eigenproblem that we analyze in this paper there are severe nonconformities, due to which established conforming theories \cite{Babuska,Boffi,Knyazev} are not applicable.  In the setting of the finite element discretizations of the vector-Laplace eigenproblem \eqref{EV1} that we study, there are the following \emph{two nonconformities} that are related to very different aspects.
 Firstly,  instead of the space of \emph{tangential} vector fields, used in the continuous problem, an extended space is used, which allows \emph{non}tangential velocity components. Furthermore, in the trace finite element approach one uses finite element polynomials defined in a small volume neighborhood of the surface. These function space extensions are the reason why one uses  \emph{penalization}  in these finite element methods. 
   A second very different source of nonconformity comes from the approximation of the exact surface: $\Gamma_h \approx \Gamma$. This we call a \emph{geometric inconsistency}. We note that such geometric inconsistencies are a key issue in the analysis of scalar surface partial differential equations, too \cite{DEreview,Bonito2019}. The former issue related to tangentiality arises only for vector-valued surface PDEs.  
   The main topic of this paper is an  error analysis of  finite element element discretizations of the vector-Laplace eigenproblem \eqref{EV1} that handles these two essential nonconformities. 
   We develop the error analysis in an abstract general framework. An  eigenproblem in the following standard Hilbert space setting is considered.  Let $ H \subset \hat H$ be two infinite dimensional Hilbert spaces, with $H$ compactly embedded in $\hat H$. Let $a:\, H\times H \to \Bbb{R}$, $b:\, \hat H\times \hat H \to \Bbb{R}$ be  bounded symmetric elliptic bilinear forms. We consider the eigenproblem: $u \in H,\, \lambda \in \Bbb{R}$ such that
   \begin{equation} \label{AA}
    a(u,v)= \lambda b(u,v) \quad \text{for all}~v \in H.
   \end{equation}
The vector-Laplace problem \eqref{EV1} can be cast in this variational form. For a nonconforming discretization of this problem we use the following setting, cf. section~\ref{sectHilbert} for more details. We introduce a ``richer''  space $\He$, an ``extension operator'' $\cE: \, H \to \He$, with a corresponding lifting (or pull back) operator $\CEl: \He \to H$, which is a left inverse of $\cE$. Furthermore, $(V_h)_{h >0}$ is a family of finite dimensional subspaces of $\He$. We study discretizations of the form: $\tilde u \in V_h$, $\tilde \lambda \in \Bbb{R}$ such that
\begin{equation} \label{AAd} 
    a_h(\tilde u,\tilde v) = \tilde \lambda b_h(\tilde u,\tilde v) \quad \text{for all}~\tilde v \in V_h.
\end{equation}
The ``nonconforming'' bilinear forms $a_h(\cdot,\cdot)$ and  $b_h(\cdot,\cdot)$ have the structural form $a_h(u,v)= \tilde a_h(u,v) +k_a(u,v)$, $b_h(u,v)=\tilde b_h(u,v)+k_b(u,v)$, $u,v \in \He$. The symmetric positive semidefinite bilinear forms $k_a(\cdot,\cdot)$ and $k_b(\cdot,\cdot)$ correspond to penalizations. These penalizations and $\tilde a_h(\cdot,\cdot)$, $\tilde b_h(\cdot,\cdot)$, have to fulfill certain consistency conditions such that $a_h(\cdot,\cdot)$ and $b_h(\cdot,\cdot)$ are ``close to'' $a(\cdot,\cdot)$ and $b(\cdot,\cdot)$, respectively. These consistency conditions, combined with a condition on the relative strength of the penalizations $k_a(\cdot,\cdot)$, $k_b(\cdot,\cdot)$, and with an approximabilty condition  for the (extended) eigenvectors in the space $V_h$ lead to error bounds both for the eigenvalue and eigenvector approximations. We will show how this general error analysis framework  can be applied to two known classes of finite element discretization methods for the eigenproblem \eqref{EV1}.

As far as we know, this is the first error analysis that applies to finite element discretizations of vector-Laplace problems \eqref{EV1} and in which nonconformities due to penalization and due to geometric inconsistencies are treated.  This analysis  leads to optimal order error bounds both for eigenvalues and eigenvectors. Furthermore, in the same spirit as in \cite{Knyazev1,Knyazev}, our estimates depend on explicitly given quantities, in our case consistency and approximability parameters.  Our analysis is \emph{sub}optimal with regard to the following aspect. It does not take into account that different eigenvectors may have different approximabilities, resulting in different levels of error for approximate eigenvalues, cf. the discussion in \cite{Knyazev}.  Our estimate for the $j$th eigenvalue error depends on approximability of all eigenvectors in the corresponding eigenspace, as well as the eigenvectors corresponding to all smaller eigenvalues. It is well-known that this is not realistic \cite{Knyazev}. We note, however, that in the setting of our applications  it is reasonable to assume that all eigenvectors corresponding to the smallest $j$ eigenvalues have comparable approximability properties.

The remainder of the paper is organized as follows. In section~\ref{Sectcontinuous} we introduce the variational formulation of the surface vector-Laplace eigenproblem \eqref{EV1}. In section~\ref{sectFEM} we recall two basic finite element discretization methods for vector-valued surface PDEs, known from the literature. Both methods use the same scalar finite element space for each of the three components of the velocity field $\bu$, and the tangential condition is weakly enforced by a penalty method. In section~\ref{sectHilbert} an abstract general analysis framework is presented. In this framework a discretization of the eigenproblem is introduced in which the penalty technique and inconsistencies due to geometry approximation are formalized. For this  abstract discrete problem error bounds for the eigenvalues  and eigenvectors are derived in the sections~\ref{sectRR} and \ref{secteigenvector}, respectively. A discussion of the main results of this abstract analysis is given in section~\ref{sectdiscussion}. In section~\ref{sectTraceFEM} the general analysis is applied to the finite element methods treated in section~\ref{sectFEM} and (optimal order) error bounds both for eigenvalues and eigenvectors are derived.  Finally, in section~\ref{sectExp} we present results of a numerical experiment that illustrates certain convergence properties.

\section{Vector-Laplace eigenproblem} \label{Sectcontinuous}
Let $\Gamma \subset \Bbb{R}^3$ be a sufficiently smooth (at least $C^2$) compact surface without boundary.  Vector  fields on $\Gamma$ are denoted by boldface symbols $\bu$, $\bv$.  
 We use the setting as in papers on surface partial differential equations, which is based on tangential calculus, e.g., \cite{hansbo2016analysis,Bonito2019}.
A tubular neighborhood of $\Gamma$ is defined by
$U_\delta := \left\lbrace x \in \mathbb{R}^3 \mid \vert d(x) \vert < \delta \right\rbrace$,
with $\delta > 0$ and $d$ the signed distance function to $\Gamma$, which we take negative in the interior of $\Gamma$.
 On $U_\delta$ we define $\bn(x) = \nabla d(x)$, the outward pointing unit normal on $\Gamma$, $\bH(x) = \nabla^2d(x)$, the Weingarten map,  $\bP = \bP(x):= \bI - \bn(x)\bn(x)^T$, the orthogonal projection onto the tangent space, $p(x) = x - d(x)\bn(x)$, the closest point projection. We assume $\delta>0$ to be sufficiently small such that the decomposition 
$
x = p(x) + d(x) \bn(x)
$
is unique for all $x \in U_{\delta}$. The constant normal extension for vector functions $\bv \colon \Gamma \to \mathbb{R}^3$ is defined as $\bv^e(x) := \bv(p(x))$, $x \in U_{\delta}$. The extension for scalar functions is defined similarly. Note that on $\Gamma$ we have $\nabla \bw^e = \nabla(\bw \circ p) = \nabla \bw^e \bP$, with $\nabla \bw := (\nabla w_1, \nabla w_2, \nabla w_3)^T \in \mathbb{R}^{3 \times 3}$ for smooth vector functions $\bw \colon U_\delta \to \mathbb{R}^3$. For a scalar function $g \colon \Gamma \to \mathbb{R}$ and a vector function $\bv \colon \Gamma \to \mathbb{R}^3$ (not necessarily tangential to $\Gamma$) we  define surface derivatives by
\begin{equation*} \begin{split}
\nabla_{\Gamma} g(x) &= \bP(x)\nabla g^e(x) \in \Bbb{R}^3, \quad x \in \Gamma, \\
\nabla_{\Gamma} \bv(x) &= \bP(x)\nabla \bv^e(x) \bP(x) \in \Bbb{R}^{3 \times 3}, \quad x \in \Gamma. 
\end{split}
\end{equation*} 
If $\bv$ is tangential to $\Gamma$, $\nabla_{\Gamma} \bv$ is the covariant derivative. 
Finally, we introduce a notation for the symmetric part of $\gradG \bv$:
\begin{equation*}
E(\bv):= \frac{1}{2} \left( \gradG \bv + \gradG \bv^T \right) \in \Bbb{R}^{3 \times 3}.
\end{equation*}
 We need the surface divergence operator for vector-valued functions $\bu \colon \Gamma \to \mathbb{R}^3$ and tensor-valued functions $\bA \colon \Gamma \to \mathbb{R}^{3\times3}$. These  are defined as
\begin{equation*} \begin{split}
\divG \bu &:= \textrm{tr} (\gradG \bu),  \\
\divG \bA &:= \left( \divG (e_1^T\bA), \divG (e_2^T\bA),\divG (e_3^T\bA)  \right)^T,
\end{split}
\end{equation*}
with $e_i$ the $i$th basis vector in $\mathbb{R}^3$.
The surface Sobolev space of $k$ times weakly differentiable  functions is denoted by $H^k(\Gamma)$, $k \in \Bbb{N}$. For vector valued functions (values in $\Bbb{R}^3$) we write $\bH^k(\Gamma):=H^k(\Gamma)^3$. We introduce a notation for the  vector valued functions that are tangential:
\begin{equation*}
\bV_T := \left\lbrace \bu \in \bH^1(\Gamma) \mid \bu \cdot \bn =0 \right\rbrace.
\end{equation*}
Endowed with the usual $H^1$ scalar product, the space $\bV_T$ is a Hilbert space.
On $\bV_T$ we define the continuous, symmetric elliptic bilinear form:
\begin{equation} \label{defa}
 a(\bu,\bv):= \int_\Gamma \tr\big( E (\bu)^T E(\bv)\big) +\bu\cdot \bv \, ds, \quad \bu,\bv \in \bV_T.
\end{equation}
Ellipticity of this bilinear form follows from a surface Korn's inequality \cite{Jankuhn1}.
We formulate a vector-Laplace eigenproblem: determine $\lambda \in \Bbb{R}$, $\bu^\ast \in \bV_T$ such that
\begin{equation} \label{varev}
  a(\bu^\ast,\bv)=\lambda (\bu^\ast,\bv)_{L^2(\Gamma)}~~\text{for all}~\bv \in \bV_T. 
\end{equation}
With $-\hat \Delta_\Gamma \bu :=-\tfrac12 \bP\divG(\gradG \bu  + (\gradG \bu)^T)$ this eigenproblem has a formulation in strong form as
\[-\hat \Delta_\Gamma \bu + \bu = \lambda \bu.
\]
 
\begin{remark} \label{RemBochner} \rm
The vector-Laplace operator $\hat \Delta_\Gamma$ differs from the Hodge  and Bochner Laplacians. The latter is defined by $\Delta_\Gamma \bu := \bP \divG (\gradG \bu)$. Finite element methods for surface Bochner Laplace problems are studied in \cite{hansbo2016analysis}. The following relation holds, cf.~\cite{Jankuhn1}:
\[
   -2 \hat \Delta_\Gamma \bu = - \Delta_\Gamma \bu - \gradG \divG \bu - K\bu.
\]
The Bocher Laplacian $- \Delta_\Gamma$ is an elliptic operator, with a smallest eigenvalue that is strictly positive. The operator $- \gradG \divG$ is positive semidefinite. The operator $-\hat \Delta_\Gamma$ is positive semidefinite and can have an eigenvalue (close to) zero due to the additional operator $K \, {\rm id}$.
\end{remark}

Killing vector fields (KVF) are (tangential) vector fields $\bu$ that are in the in the kernel of $E$, i.e., $E(\bu)=0$. These KVF are studied in  differential geometry \cite{Petersen}, in literature on (approximate) isometries in computer graphics \cite{Azencot2015,BenChen2010,Tao2016} and in papers that treat surface (Navier-)Stokes equations \cite{Kobaetal_QAM_2017,miura2017singular,Jankuhn1}. Note that KVF are in the kernel of the operator $-\hat \Delta_\Gamma$. To obtain an elliptic operator we add a shift, i.e., we consider $-\hat \Delta_\Gamma + I$, which does not change the eigenfunctions and shifts all eigenvalues by $1$. 

In the remainder of this paper we consider the  vector-Laplace eigenvalue problem  \eqref{varev} and analyze  finite element discretization methods for this problem.   We consider the situation that we want to approximate the eigenspaces corresponding to \emph{a fixed (small) number of the smallest eigenvalues}.

\section{Finite element discretization methods} \label{sectFEM}
 The most popular and conceptually simplest method for discretization of vector-valued surface PDEs is a generalization of the Dziuk-Elliott surface finite element method (SFEM), in which  standard continuous parametric Lagrange finite elements are used to approximate a vector field on the surface, and the tangent condition is enforced weakly using a penalization term. Another approach that uses the same penalization technique is based on trace finite elements (TraceFEM).

We outline the basic structure of both methods. We first consider a variant of the SFEM for vector-valued surface PDEs.
Several options for constructing higher order parametric surface approximations have appeared. We recall one basic variant.
We assume a piecewise triangular (quasi-uniform) approximation $\Gamma^{\text{lin}}$ of the surface $\Gamma$, with mesh size parameter denoted by $h$ and vertices $x_j$, $j=1,\ldots,n_h$, that is assumed to be sufficiently close to $\Gamma$: ${\rm dist}( x_j,\Gamma )\leq c h^2$, $1 \leq j\leq n_h$. The nodal Lagrange basis functions of piecewise polynomials of degree $k_g$ on $\Gamma^{\text{lin}}$ are denoted by $\phi_i$. Recall that $p$ denotes the closest point projection onto $\Gamma$. We define
\begin{equation} \label{kg1}
 \Gamma_h^{k_g}:= \{\, L(x)~,~ x \in \Gamma^{\text{lin}}\,\}, ~~L(x):=\sum_i p(x_i) \phi_i(x). 
\end{equation}
A corresponding (higher order) finite element space is defined by
\begin{equation} \label{defVh1} \begin{split}
  V_h^{k,k_g}&:= \left\lbrace \, v \in H^1(\Gamma_h^{k_g})~\mid~v=\bar v \circ L^{-1},~~\bar v_{|T} \in P_k,~\forall ~T \in \Gamma^{\text{lin}}\, \right\rbrace, \\
  \bV_h^{k,k_g} &:=(V_h^{k,k_g})^3. \end{split}
\end{equation}
Note that $k_g$ denotes the degree of the polynomials used in the parametric mapping $L$ and $k$ the degree of the polynomials used in the finite element space. 
To simplify the notation we delete the superscript $k_g$ and write $\Gamma_h$, $V_h^{k}$, $\bV_h^k$. 

In  finite element methods for surface vector-Laplace and surface (Navier-)Stokes equations it is convenient to allow a possibly nontangential velocity field $\bu \in \Bbb{R}^3$. We emphasize that in the remainder of this section \emph{vector fields $\bu$, $\bu_h$ are not necessarily tangential}. A vector field on $\Gamma$ is decomposed  in tangential and normal components as $\bu=\bu_T+ (\bu\cdot \bn)\bn =: \bu_T + u_N \bn$. For the symmetric gradient the identity $E(\bu) = E(\bu_T)+u_N \bH$ holds, hence
\begin{equation} \label{idE}
  E(\bu_T)= E(\bu) - u_N \bH.
\end{equation}
In the finite element method we use a penalty term, denoted by $k_h(\cdot,\cdot)$ below, to enforce a discrete solution $\bu_h$ to be ``almost tangential''.
 We define obvious discrete variants of derivatives and bilinear forms used in the surface vector-Laplace eigenvalue problem \eqref{varev}, in particular $\bP_h:=\bI- \bn_h \bn_h^T$, with $\bn_h$ the unit normal on $\Gamma_h$, and (for $\bu \in H^1(\Gamma_h)^3$ smoothly extended off $\Gamma_h$):
\begin{equation}  \label{defall} \begin{split}
 \gradGh \bu(x) &:= \bP_h(x) \nabla \bu(x) \bP_h(x),\quad x \in \Gamma_h, \\ 
  E_h(\bu)&  := \frac12 \big(\gradGh \bu + \gradGh \bu^T\big), \quad E_{T,h}(\bu):=E_h(\bu) - u_N \bH_h, \\
 a_h(\bu,\bv) &:= \int_{\Gamma_h} \tr (E_{T,h}\big(\bu)^T E_{T,h}(\bv)\big)\, ds_h + \int_{\Gamma_h}\bP_h \bu_h \cdot \bP_h \bv_h \, ds_h, \\
 k_h(\bu,\bv)&:= \eta \int_{\Gamma_h} (\bu \cdot \tilde{\bn}_h) (\bv \cdot \tilde{\bn}_h)  \, ds_h,\quad \eta\sim h^{-2}.
\end{split}
\end{equation}
The scaling $\eta\sim h^{-2}$ of the penalty parameter is based on analysis from the literature. 
The curvature tensor $\bH_h$ is an approximation of the  Weingarten mapping $\bH$. The vector $\tilde{\bn}_h$   used in the penalty term is a sufficiently accurate approximation of the exact normal $\bn$. It should be  of at least one order higher accuracy than the normal approximation $\bn_h$, cf. \cite{hansbo2016analysis,Jankuhn2} and section~\ref{sectTraceFEM}.
The finite element discretization of the eigenvalue  problem  \eqref{varev} is as follows: determine $\lambda_h \in \Bbb{R}$,
 $\bu_h \in  \bV_{h}^{k}$ such that
\begin{equation} \label{P2hh}\begin{split} 
 A_h(\bu_h,\bv_h)& =\lambda_h B_h(\bu_h,\bv_h) \quad \text{for all}~\bv_h \in \bV_{h}^{k}, \\
 A_h(\bu_h,\bv_h)&: = a_{h}(\bu_h,\bv_h) + k_h(\bu_h, \bv_h), \\ B_h(\bu_h,\bv_h)&:=(\bu_h, \bv_h)_{L^2(\Gamma_h)}. 
\end{split} \end{equation}

We now briefly address the TraceFEM for vector surface PDEs. 
We assume that the surface $\Gamma$ is respresented as the zero level of a smooth level set function $\phi$ that is defined on a polygonal domain $\Omega \subset \Bbb{R}^3$ that contains the surface $\Gamma$. 
Let $\{ \mathcal{T}_h \}_{h>0}$ be a family of (quasi-uniform) shape regular tetrahedral triangulations of  $\Omega$.  
We construct a geometry approximation $\Gamma_h \approx \Gamma$ that is based on a parametric transformation of the \emph{triangulation}. As input for this transformation we assume a sufficiently accurate approximation $\phi_h \approx \phi$ that is continuous and piecewise polynomial of degree $k_g$ on $\cT_h$. 
Based on the piecewise \emph{linear} nodal interpolation of $\phi_h$, which is denoted by $\hat{\phi}_h$, we define the low order piecewise planar geometry approximation $
\Gamma^{\text{lin}} := \{ x \in \Omega \mid \hat{\phi}_h (x) = 0\}$. 
The  tetrahedra $T \in \cT_h$ that have a nonzero intersection with $\Gamma^{\text{lin}}$ are collected in the set denoted by $\T_h^\Gamma$. The domain formed by all tetrahedra in $\T_h^\Gamma$ is denoted by $\OGamma:= \{ x \in T \mid T \in \T_h^\Gamma \}$. Let $\Theta_h^{k_g}:\OGamma \to \Omega $ be the \emph{mesh} transformation of order $k_g$ as defined in \cite{grande2017higher}. This is a vector valued mapping and each of its components is a standard Lagrangian finite element function on $\Omega_h^\Gamma$ of degree $k_g$.   
We denote the transformed cut mesh domain by $\Omega^{\Gamma}_\Theta := \Theta_h^{k_g}(\Omega^\Gamma_h)$ and the
approximation of $\Gamma$ is defined as
\begin{equation} \label{kg2}
\Gamma_{h}^{k_g} := \Theta_h^{k_g}(\Gamma^{\text{lin}}) = \left\lbrace x \mid \hat{\phi}_h((\Theta_{h}^{k_g})^{-1}(x)) = 0 \right\rbrace \subset \Omega^{\Gamma}_\Theta.
\end{equation}
The finite element  space is defined by
\begin{equation} \label{TraceFEspace} \begin{split} 
V_{h,\Theta}^{k,k_g} & := \left\lbrace v \in H^1(\Omega_h^\Gamma) \mid v=\bar v \circ (\Theta_h^{k_g})^{-1},~\bar v_{|T} \in P_k,~\forall ~T \in \cT_h^\Gamma   \right\rbrace, \\
\bV_{h,\Theta}^{k,k_g} &:= (V_{h,\Theta}^{k,k_g})^3. \end{split}
\end{equation}
Again we delete the superscript $k_g$ and write $V_{h,\Theta}^k$, $\bV_{h,\Theta}^{k}$.
For treating the tangential constraint we use the same penalty term $k_h(\cdot,\cdot)$ as above. Since we use an unfitted finite element method, we need a stabilization that eliminates instabilities caused by the small cuts. For this we use the so-called ``normal derivative volume stabilization'' \cite{grande2017higher}:
\begin{align*} 
  s_h(\bu,\bv)& :=  \int_{\Omega_{\Theta}^{\Gamma}} (\nabla \bu \bn_h) \cdot (\nabla \bv \bn_h)  \, dx.
\end{align*}
The TraceFEM discretization of the eigenvalue  problem  \eqref{varev} is as follows: determine $\lambda_h \in \Bbb{R}$,
 $\bu_h \in  \bV_{h,\Theta}^{k}$ such that
\begin{equation} \label{P2h}\begin{split} 
 A_h(\bu_h,\bv_h)& =\lambda_h B_h(\bu_h,\bv_h) \quad \text{for all}~\bv_h \in \bV_{h,\Theta}^{k}, \\
 A_h(\bu_h,\bv_h)&: = a_{h}(\bu_h,\bv_h) + k_h(\bu_h, \bv_h) + \rho_a  s_h(\bu_h,\bv_h),\\ B_h(\bu_h,\bv_h)&:=(\bu_h, \bv_h)_{L^2(\Gamma_h)} + \rho_b s_h(\bu_h,\bv_h), 
\end{split} \end{equation}
with $a_{h}(\cdot,\cdot),\, k_h(\cdot,\cdot)$ as in \eqref{defall}. Appropriate scalings for the stabilization parameters $\rho_a$, $\rho_b$ are, cf. section~\ref{sectTraceFEM},
\begin{equation} \label{scalings}
 \rho_a \sim h^{-1},\quad \rho_b\sim h.
\end{equation}
The main topic of this paper is an analysis of the discretization accuracy of the eigenproblems \eqref{P2hh}, \eqref{P2h}. We will present an analysis in a general abstract setting, which applies to discretization methods as the ones above.  

 Compared to the continuous problem \eqref{varev}, both discrete problems above have  two \emph{nonconformities} that are related to very different aspects:
 \begin{itemize} \item Instead of the space of \emph{tangential} vector fields $\bV_T$, used in the continuous problem, an \emph{extended} space is used, which allows \emph{non}tangential velocity components. For the TraceFEM, in addition we use functions that are defined not only on $\Gamma$ (or $\Gamma_h$), but in a small (volume) neighborhood $\Omega_h^\Gamma$. These function space extensions are the reason why one uses the \emph{penalization} $k_h(\cdot,\cdot)$ in \eqref{P2hh} and $k_h(\cdot,\cdot)$ and $s_h(\cdot,\cdot)$ in \eqref{P2h}. We interprete the stabilization $s_h(\cdot,\cdot)$ as a \emph{penalization} of variation of functions in the direction normal to the (approximate) surface. 
  \item 
   A very different source of nonconformity comes from the approximation of the exact surface: $\Gamma_h \approx \Gamma$. This we call a \emph{geometric inconsistency}. 
  \end{itemize} 
  We note that the issue related to the tangential condition does not occur in scalar Laplace-Beltrami eigenproblems.
 We will analyze the effect of both \emph{penalization} and \emph{geometric inconsistency} on the accuracy of the discrete eigenproblem.
 To clearly identify the effects of these two types of nonconformities on the discretization error 
we introduce an abstract analysis framework.
Application of the general results to the specific  discretizations \eqref{P2hh} and  \eqref{P2h} is treated in section~\ref{sectTraceFEM}. 

\section{Abstract Hilbert space setting} \label{sectHilbert}
We recall the usual framework of symmetric elliptic eigenvalue problems in Hilbert spaces. Let $ H \subset \hat H$ be two infinite dimensional Hilbert spaces, with $H$ compactly embedded in $\hat H$. Let $a:\, H\times H \to \Bbb{R}$ be a bounded symmetric elliptic bilinear form. For simplicity we equip $H$ with the energy norm $\Ecnorm{u}:=a(u,u)^\frac12$. The scalar product on $\hat H$ is denoted by $b(\cdot,\cdot)$, with norm denoted by $\|u\|=b(u,u)^\frac12$. The spectrum of $a(\cdot,\cdot)$ consists of an infinite sequence $0 < \lambda_1 \leq \lambda_2 \leq \ldots$ of eigenvalues of finite multiplicity, tending to infinity, and a corresponding sequence of eigenvectors $u_1,u_2, \ldots \in H$, such that
\begin{equation} \label{eqevproblem}
  a(u_k,u_\ell)= \lambda_k \delta_{k\ell},~~b(u_k,u_\ell) = \delta_{k\ell}, \quad k,\ell \in \Bbb{N},
\end{equation}
and $\delta_{k\ell}$ the Kronecker delta. 
The aim is to approximate eigenpairs $(\lambda_k,u_k)$, $k=1,\ldots, k_{\max}$ for a \emph{ fixed small number} $k_{\max}$. We introduce another space $\He$ that contains an infinite family of finite dimensional discretization spaces $(V_h)_{h >0}$. The paramater $h$ has strictly positive values with accumulation point 0.  To connect the spaces we assume  a linear injective extension (or embedding)  operator $\cE: H \to \He$. Furthermore there is a ``lifting'' operator that maps elements in $\He$ back to $H$. This operator may depend on the particular discretization and therefore we use a subscript $h$ and denote this lifting by $\cE_{h}^{-\ell} : \He \to H$. We assume that it  is a left inverse of $\cE$, 
i.e., $\cE_{h}^{-\ell} \cE={\rm id}_H$.  For the extension we use the notation $\cE u=:u^e$ (``extension/embedding'' of $u\in H$ in $\He$). The subspace of $\He$ consisting of extended eigenvectors is denoted by $U_j^e:={\rm span}\{u_1^e, \ldots,u_j^e\}$. 
\begin{remark} \label{remexample1}\rm
 For the vector-Laplace eigenvalue  problem we take $H=\bV_T$, $\hat H=\{\bu \in L^2(\Gamma)^3~|~ \bu \cdot \bn=0\,\}$, $a(\cdot,\cdot)$ as in \eqref{defa} and $b(\bu,\bv)= \int_\Gamma \bu \cdot \bv \, ds$. As  extended space $\He$  we take:
 \begin{align*}
   \He & = H^1(\Gamma_h)^3 \quad \text{for SFEM}, \\
   \He & = \{\, \bu \in H^1(\Omega_{\Theta}^\Gamma)^3~|~\bu_{|\Gamma} \in H^1(\Gamma)^3~ \text{and}~\bu_{|\Gamma_h} \in H^1(\Gamma_h)^3\,\} \quad \text{for TraceFEM}.
 \end{align*}
Note that we do not have a tangential condition for vector functions in $\He$.  For the extension operator $\cE$ we take the constant extension of functions on $\Gamma$ to $\Gamma_h$ (for SFEM) or to $\Omega_{\Theta}^\Gamma$ (for TraceFEM) along normals on $\Gamma$. The pull back operator $\cE_h^{-\ell} $ is given by $\cE_h^{-\ell} \bu = \bP (\bu_{|\Gamma_h})^{\ell}$, where  $(\bu_{|\Gamma_h})^{\ell}$ denotes the usual pull back operator used in the analysis of surface PDEs \cite{DEreview}, namely the constant extension of functions on $\Gamma_h$ along normals on $\Gamma$ to obtain values on $\Gamma$. Note that for the discretization spaces we then have $V_h \subset \He$.
\end{remark}

\subsection{Discrete eigenproblem with penalization and  inconsistency} \label{sectabstractdiscrete}
For the discretization of \eqref{eqevproblem} in $V_h$ we introduce the following general abstract setting. 
 We assume    bilinear forms  $ a_h(\cdot,\cdot)$, $ b_h(\cdot,\cdot)$ on $H^{\rm ex} \times \He$ of the form
 \begin{equation} \label{defblfA}
  \begin{split}
    a_h(u,v)&:= \tilde a_h(u,v) + k_a(u,v), \quad u,v \in \He, \\
   b_h(u,v)&:= \tilde  b_h(u,v) + k_b(u,v), \quad u,v \in \He,
  \end{split}
 \end{equation}
 with $\tilde a_h(\cdot,\cdot)$, $\tilde b_h(\cdot,\cdot)$, $k_a(\cdot,\cdot)$, $k_b(\cdot,\cdot)$ symmetric positive \emph{semi}definite bilinear forms on $\He$. In these bilinear forms the parts $\tilde a_h(\cdot,\cdot)$ and $\tilde b_h(\cdot,\cdot)$ are approximations of $a(\cdot,\cdot)$ and $b(\cdot,\cdot)$, respectively, and $k_a(\cdot,\cdot)$, $k_b(\cdot,\cdot)$, correspond to  penalizations. These penalty bilinear forms may depend on $h$, but to simplify the notation, this dependence is not made explicit. 
 The corresponding seminorms are denoted by 
 \[ a_h(\cdot,\cdot)=\Enorm{\cdot}^2, \quad b_h(\cdot,\cdot)=\|\cdot\|_h^2.
 \]
 We \emph{assume} that $a_h(\cdot,\cdot)$ and $b_h(\cdot,\cdot)$ are positive definite on $V_h\times V_h$. The discrete eigenproblem is as follows: determine $\tilde \lambda \in \Bbb{R}$, $ \tilde u \in V_h$ such that
\begin{equation} \label{discrev}
  a_h( \tilde u,\tilde v)= \tilde \lambda b_h(\tilde u,\tilde v) \quad \text{for all}~\tilde v \in V_h.
\end{equation}
The solutions of this problem form an orthogonal basis of eigenvectors $\tilde u_k\in V_h$, $k=1,\ldots,n:=\dim(V_h)$, with corresponding eigenvalues $0< \tlam_1 \leq \tlam_2 \leq \ldots \leq \tlam_n$, such that
\begin{equation} \label{approxEVPa}
 a_h(\tilde u_k,\tilde u_\ell)=\tlam_k \delta_{k\ell},\quad b_h(\tilde u_k,\tilde u_\ell)=\delta_{k\ell}, \quad 1 \leq k,\ell \leq n.
\end{equation}
We define the scaled eigenvectors $\hat u_k:= \tlam_k^{-\frac12} \tilde u_k$, hence, $\Enorm{\hat u_k}=1$, and use the notation $\tilde U_j={\rm span}\{\tilde u_1, \ldots, \tilde u_j\}$.
\begin{remark} \label{remexample2} \rm 
We briefly comment on how the discretizations treated in section~\ref{sectFEM} fit in this setting, cf. Remark~\ref{remexample1}. 
 For the SFEM \eqref{P2hh} we take $\tilde a_h(\cdot,\cdot)=a_h(\cdot,\cdot)$ as in \eqref{defall}, $\tilde b_h(\bu,\bv)= \int_{\Gamma_h} \bP_h \bu \cdot \bP_h \bv \, ds_h$ and the penalty bilinear forms are given by $k_a(\bu,\bv)=k_h(\bu,\bv)$ as in \eqref{defall}, $k_b(\bu,\bv)= \int_{\Gamma_h}(\bu \cdot \bn_h) (\bv \cdot \bn_h) ds_h$. This implies $\tilde b_h(\bu,\bv)+k_b(\bu,\bv)= B_h(\bu,\bv)$ as in \eqref{P2hh}. For the TraceFEM we use the same bilinear forms $\tilde a_h(\cdot,\cdot)$ and $\tilde b_h(\cdot,\cdot)$. The penalty bilinear forms are $k_a(\bu,\bv)=k_h(\bu,\bv) + \rho_a \int_{\Omega_\Theta^\Gamma} (\nabla \bu \bn_h) \cdot (\nabla \bv \bn_h)\, dx$, $k_b(\bu,\bv)=\int_{\Gamma_h} (\bu \cdot \bn_h) (\bv \cdot \bn_h)\, ds_h + \rho_b \int_{\Omega_\Theta^\Gamma} (\nabla \bu \bn_h) \cdot (\nabla \bv \bn_h)\, dx$. Note that the normal derivative volume stabilization terms ($s_h(\cdot,\cdot)$ above) are part of the penalty bilinear forms.
\end{remark}

In the analysis below we use several  orthogonal projections, that we now introduce.
The orthogonal projection w.r.t. the energy seminorm is denoted by $P_h:\, \He \to V_h$, and defined by
\begin{equation} \label{Phdef}
  a_h(w,v_h)=a_h(P_hw,v_h) \quad \text{for all}~v_h \in V_h.
\end{equation}
This projection is well-defined due to the assumption that $a_h(\cdot,\cdot)$ is a scalar product on the subspace $V_h$. 
For this projection we have the representation 
\begin{equation} \label{formprojection} P_hw= \sum_{i=1}^n a_h(w,\hat u_i) \hat u_i . \end{equation}
As stated above, we restrict to a small number $k_{\max} \ll n={\rm dim}(V_h)$ of the smallest eigenvalues and corresponding eigenvectors.  We assume that elements in the space $U_{k_{\max}}^e$ can be accurately approximated in the finite dimensional space $V_h$. More specifically, we assume that  ${\rm dim}(P_h(U_{k_{\max}}^e))=k_{max}$ holds. This implies
\begin{equation} \label{deftheta}
 \Theta_{h,j}:= \max_{w \in U_j^e} \frac{\Enorm{w- P_hw}}{\Enorm{w}} <1, \quad 1 \leq j \leq k_{\max}.
\end{equation}
The parameter $\Theta_{h,j}$ quantifies the \emph{approximability of the (extended) eigenvectors in the discretization space} $V_h$.

We also need orthogonal projections w.r.t. $a_h(\cdot,\cdot)$ and $b_h(\cdot,\cdot)$ onto $U_j^e$, $1 \leq j \leq \kmax$. For these to be well-defined we assume that $a_h(\cdot,\cdot)$ and $b_h(\cdot,\cdot)$ are inner products on $U_{\kmax}^e$. This property will follow from assumptions that are introduced further on, cf. Remark~\ref{remnorm}. An orthogonal projection $\Paj: \He \to U_j^e$ is uniquely defined via
\begin{equation} \label{projectionahj}
 \Enorm{v - \Paj v}= \min_{w \in U_j^e}\Enorm{v-w}, \quad v \in \He.
\end{equation}
Similarly we define $\Pbj$, the orthogonal projection w.r.t. $\|\cdot\|_h$ onto $U_j^e$. Note that for these projections we do not have an analogon of the formula \eqref{formprojection}, because $u_i^e$, $i=1,2,\ldots$, is not necessarily orthogonal w.r.t. $a_h(\cdot,\cdot)$, $b_h(\cdot,\cdot)$.
  
These projections are illustrated in Fig.~\ref{figprojection}.
\begin{figure}[ht!]
\begin{center}
\scalebox{.6}{%\documentclass[tikz,border=10pt,multi]{standalone}
%\usepackage{tikz}
%\usepackage{bm}
%\usetikzlibrary{...}

%\begin{document}

\begin{tikzpicture}
%Coordsystem
\draw [->] (-1, 0) -- (10, 0);
\draw [->] (0, -1) -- (0, 10);
\node [right] at (10, 0) { \Large{$V_h$} };

\draw [-] (0, 0) -- (5, 10);
\node [right] at (5, 10) { \Large{$U_j^e$} };

%Points
\fill [black] (3,6)  circle  (.05cm) node [left] {\Large{$P_{a_{h,j}}v$}};
\fill [black] (6,0)  circle  (.05cm) node [below] {\Large{$P_hv$}};
\fill [black] (3.5,7)  circle  (.05cm) node [left] {\Large{$P_{b_{h,j}}v$}};
\fill [black] (7.5,0)  circle  (.05cm) node [below] {\Large{$Q_hv$}};
\fill [black] (6,4.5)  circle  (.05cm) node [right] {\Large{$v$}};

%Lines
\draw [dashed] (3, 6) -- (6, 4.5) -- (6, 0);
\draw [dotted] (3.5,7) -- (6, 4.5) -- (7.5, 0);

%Angles, sqrt(2)=1.41421356237, sqrt(5)=2.2360679775
\draw [-] (6, 0.25) -- (6.25, 0.25) -- (6.25 ,0);
\draw [-] (3 + 0.25/2.2360679775, 6+0.25/2.2360679775*2) -- (3 + 0.25/2.2360679775 + 0.25/2.2360679775 *2, 6+0.25/2.2360679775*2-0.25/2.2360679775) -- (3 + 0.25/2.2360679775 *2, 6-0.25/2.2360679775);
\draw [-] (3.5 + 0.25/2.2360679775, 7+0.25/2.2360679775*2) .. controls (3.5 + 0.25/2.2360679775 + 0.25/1.41421356237, 7+0.25/2.2360679775*2 -0.25/1.41421356237) and (3.5 + 0.25/2.2360679775 + 0.25/1.41421356237, 7+0.25/2.2360679775*2 -0.25/1.41421356237) .. (3.5 + 0.25/1.41421356237, 7-0.25/1.41421356237);
\draw [-] (7.5+0.25,0) .. controls (7.5+0.25-0.25*0.31622,0.25*0.94868) and (7.5+0.25-0.25*0.31622,0.25*0.94868) .. (7.5-0.25*0.31622,0.25*0.94868);

%Text, sqrt(10)=3.16227766017 
\node[] at (8,9) {\Large{$H^{\rm ex}$}};
\draw [-] (9, 6.5) -- (9, 6) -- (9.5, 6);
\draw [-] (9, 6.25) -- (9.25, 6.25) -- (9.25, 6.0);
\node[align=left] at (11.2,6) {{\Large orthogonal}\\\Large w.r.t. $a_h(\cdot,\cdot)$};
\draw [-] (9-0.5/3.16227766017, 3+0.5*3/3.16227766017) -- (9, 3) -- (9.5, 3);
%\draw [-] (9,3.25) .. controls (9.5+0.25-0.25*0.31622,3+0.25*0.94868) and (9.5+0.25-0.25*0.31622,3+0.25*0.94868) .. (9.5-0.25*0.31622,3);
\draw [-] (9+0.25,3) .. controls (9+0.25-0.25*0.31622,3+0.25*0.94868) and (9+0.25-0.25*0.31622,3+0.25*0.94868) .. (9-0.25*0.31622,3+0.25*0.94868);
\node[align=left] at (11.2,3) {{\Large orthogonal}\\ \Large w.r.t. $b_h(\cdot,\cdot)$};

\end{tikzpicture}

%\end{document}
}
\caption{Projections w.r.t. $a_h(\cdot,\cdot)$ and $b_h(\cdot,\cdot)$ onto $U_j^e$ and $V_h$\label{figprojection}}
\end{center}
\end{figure}
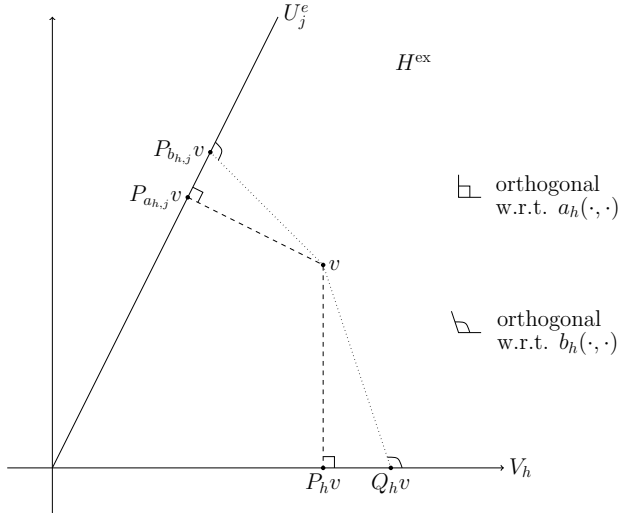

\section{Eigenvalue error analysis of discrete eigenproblem with penalization and  inconsistency} \label{sectRR}
In this section we present  an error analysis for the eigen\emph{value} approximation in the discretization \eqref{discrev}. An error analysis for eigen\emph{vectors} is given in section~\ref{secteigenvector}.
To analyze the errors coming from two very different nonconformities (penalization and  inconsistency) we first restrict to the situation with penalization only (section~\ref{sectpenal}) and then treat the general case (section~\ref{sectgeneral}).
\subsection{Analysis of discretization with penalization} \label{sectpenal}
 To avoid inconsistency, we introduce for the bilinear forms in \eqref{defblfA} the following (strongly) simplifying assumptions:
 \begin{equation} \label{conscond} \begin{split}
  \tilde a_h(u,v) & =a(\cE_h^{-\ell} u, \cE_h^{-\ell} v),~~ \tilde b_h(u,v)= a(\cE_h^{-\ell}  u, \cE_h^{-\ell} v), \quad\text{for all}~ u,v \in \He, \\
   k_a(u^e,v)&=k_a(v,u^e)=0 \quad \text{for all}~u \in H, v \in \He,\\
   k_b(u^e,v)&=k_b(v,u^e)=0 \quad \text{for all}~u \in H, v \in \He.
  \end{split}
 \end{equation}
This implies that $a_h(\cdot,\cdot)$, $b_h(\cdot,\cdot)$ are positive definite on the subspace $\cE(H)$.  
Hence, $\Enorm{\cdot}$ and  $\|\cdot\|_h$ are norms on the subspace $\cE(H)$. 
\begin{remark}  \label{remexample3} \rm 
 Perturbed versions of the SFEM and TraceFEM fulfill these assumptions. For the former we choose the discretization space $V_h$ as the space defined in \eqref{defVh1} lifted to the exact surface $\Gamma$ (by constant extension of functions along the normals $\bn$). Hence, $V_h \subset \He=H^1(\Gamma)^3 \supset \bV_T =H$. In all integrals the surface approximation $\Gamma_h$ is replaced by $\Gamma$ and $\bP_h$ and $\bn_h$ are replaced by $\bP$ and $\bn$. Clearly this yields a discretization that is not feasible in practice. We consider it here, because with this choice there are no geometric inconsistencies and the conditions \eqref{conscond} are satisfied. 
 A variant of the TraceFEM  without geometric errors is as follows. We take the finite element space $ \bV_{h,\Theta}^{k}$ as in ~\eqref{TraceFEspace}. Alternatively we can also take this space with geometry mapping $\Theta_h^{k_g}={\rm id}$. 
 In all integrals the surface approximation $\Gamma_h$ is replaced by $\Gamma$, and $\bP_h$ and $\bn_h$ are replaced by $\bP$ and $\bn$.
 Again, we obtain a method  that is not feasible.  One can check that all conditions introduced in \eqref{conscond} above are satisfied.
\end{remark}

Due to the assumptions \eqref{conscond}  the penalized bilinear forms $a_h(\cdot,\cdot)$ and  $b_h(\cdot,\cdot)$ have the following elementary property:
\begin{equation} a_h(u,v)=a(\CEl u,\CEl v), \quad b_h(u,v)=b(\CEl u,\CEl v)\quad \text{for all}~u,v \in \cE(H).\label{consprop}
\end{equation}
This implies that the extensions of the \emph{exact} eigenvectors are orthogonal w.r.t. $a_h(\cdot,\cdot)$ and $b_h(\cdot,\cdot)$:
\[
  a_h(u_k^e,u_\ell^e)= \lambda_k \delta_{k, \ell}, \quad b_h(u_k^e,u_\ell^e)=  \delta_{k, \ell}, ~k,\ell \in \Bbb{N}.
\]
Due to this, for the orthogonal projections $\Paj$, $\Pbj$, cf. \eqref{projectionahj} we have the representations
\begin{equation} \label{proja}
P_{a_h,j}w=\sum_{i=1}^j \lambda_i^{-1} a_h(w,u_i^e)u_i^e,\quad P_{b_h,j}w=\sum_{i=1}^j  b_h(w,u_i^e)u_i^e.
\end{equation}
Using $a_h(w,u_i^e)=a(\cE_h^{-\ell}w,u_i)=\lambda_i b(\cE_h^{-\ell}w,u_i)=\lambda_i b_h(w,u_i^e)$ we get the fundamental relation
\begin{equation} \label{fundrelation}
      P_{a_h,j}=P_{b_h,j}.
\end{equation}
In this situation we obtain error bounds for the eigenvalues that are the same as classical bounds from the literature for the Galerkin case, i.e., if there is no penalization and $V_h \subset H= \He$.
\begin{theorem} \label{thmEV}
We assume:
\begin{equation} \label{assstab}
 2\lambda_{k_{\max}} (1- \Theta_{k_{\max}}^2)^{-1} k_b(v,v) \leq k_a(v,v) \quad \text{for all}~v \in \tilde U_{k_{\max}}.
\end{equation}
 The following holds:
 \begin{equation} \label{resEVmain}
  0 \leq \frac{\tilde \lambda_j - \lambda_j}{\tilde \lambda_j} \leq \Theta_{h,j}^2, \quad 1 \leq j \leq k_{\max}.
 \end{equation}
\end{theorem}
\begin{proof}
For deriving the upper bound we use \eqref{consprop} and standard arguments, e.g.~\cite{Dyakonov} based on the Courant-Fischer eigenvalue characterization. Take $1 \leq j \leq k_{\max}$. Define $P_h(U_j^e):= W_j \subset V_h$. We have ${\rm dim}(W_j)=j$ and $\Paj:   W_j \to U_j^e$ is an isomorphism. Using this we get
\begin{align*}
  \lambda_j &= \max_{v \in U_j} \frac{a(v,v)}{b(v,v)} =\max_{v \in U_j^e} \frac{a_h(v,v)}{b_h(v,v)}\\ &= \max_{w \in W_j} \frac{a_h(\Paj w,\Paj w)}{b_h(\Paj w, \Paj w)}=
   \max_{w \in W_j} \frac{a_h( w, w)}{b_h( w,  w)} \cdot \frac{\Enorm{\Paj w}^2}{\Enorm{w}^2} \cdot \frac{\|w\|_h^2}{\|\Paj w\|_h^2}.
\end{align*}
Elementary properties of orthogonal projections yield $\frac{\Enorm{\Paj w}^2}{\Enorm{w}^2} \geq 1- \Theta_{h,j}^2$. From \eqref{fundrelation} we
get $\|\Paj w\|_h = \|\Pbj w\|_h \leq \|w\|_h$. Using these estimates and the Courant-Fischer theorem we get
\[
   \lambda_j \geq (1- \Theta_{h,j}^2)  \max_{w \in W_j} \frac{a_h( w, w)}{b_h( w,  w)} \geq (1- \Theta_{h,j}^2) \tilde \lambda_j,
\]
which yields the upper bound in \eqref{resEVmain}. We now derive the lower bound. First note that from \eqref{assstab} and the upper bound in  \eqref{resEVmain} it follows that
\begin{equation} \label{assstab1}
 2 \tilde \lambda_{k_{\max}} k_b(v,v) \leq k_a(v,v) \quad \text{for all}~v \in \tilde U_j.
\end{equation}
We  have
 \[
   \tilde \lambda_j = \max_{v \in \tilde U_j} \frac{a_h(v,v)}{b_h(v,v)} = \max_{v \in \tilde U_j} \frac{a(\CEl v ,\CEl v)+ k_a(v,v)}{b(\CEl v,\CEl v) + k_b(v,v)} .
 \]
We now show that $\cE_h^{-\ell}: \tilde U_j \to H$ is injective. Take $0 \neq v \in \tilde U_j$. Assume that $\cE_h^{-\ell} v =0$. This implies
\[
 \tilde \lambda_{k_{max}} \geq \tilde \lambda_j \geq \frac{a(\cE_h^{-\ell} v ,\cE_h^{-\ell} v)+ k_a(v,v)}{b(\cE_h^{-\ell} v,\cE_h^{-\ell} v) + k_b(v,v)} = \frac{ k_a(v,v)}{k_b(v,v)},
\]
hence, $k_b(v,v) > 0$ and $\tilde \lambda_{k_{max}} k_b(v,v) \geq k_a(v,v)$, which contradicts \eqref{assstab1}. For arbitrary $v \in \tilde U_j$ we have $a_h(v,v) \leq \tilde \lambda_j b_h(v,v)$. This implies
\[
  a(\cE_h^{-\ell} v,\cE_h^{-\ell} v) + k_a(v,v) - \tilde \lambda_j k_b(v,v) \leq \tilde \lambda_j b(\cE_h^{-\ell} v,\cE_h^{-\ell} v), 
\]
which, due to \eqref{assstab1}, implies
\begin{equation} \label{hulp4}
  a(\cE_h^{-\ell} v,\cE_h^{-\ell} v)\leq \tilde \lambda_j b(\cE_h^{-\ell} v,\cE_h^{-\ell} v).
\end{equation}
Using this and \eqref{assstab1} yields, for $v \in \tilde U_j$, with $k_b(v,v) >0$:
\[
  \frac{a(\CEl v,\CEl v)}{b(\CEl v,\CEl v)} \leq \tilde \lambda_j \leq \tilde \lambda_{k_{\max}} < \frac{k_a(v,v)}{k_b(v,v)},
\]
and thus for all $v \in \tilde U_j$, $v \neq 0$:
\[
  \frac{a(\cE_h^{-\ell} v,\cE_h^{-\ell} v) +k_a(v,v)}{b(\cE_h^{-\ell} v,\cE_h^{-\ell} v)+k_b(v,v)} \geq \frac{a(\cE_h^{-\ell} v,\cE_h^{-\ell} v) }{b(\cE_h^{-\ell} v,\cE_h^{-\ell} v)}  
\]
holds.
Using this, the injectivity of $\cE_h^{-\ell}$ on $\tilde U_j$ and the Courant-Fischer theorem we finally obtain
\[
  \tilde \lambda_j = \max_{v \in \tilde U_j} \frac{a_h(v,v)}{b_h(v,v)} = \max_{v \in \tilde U_j} \frac{a(\cE_h^{-\ell} v,\cE_h^{-\ell} v) +k_a(v,v)}{b(\cE_h^{-\ell} v,\cE_h^{-\ell} v)+k_b(v,v)} \geq \max_{w \in \cE_h^{-\ell} (\tilde U_j)} \frac{a(w,w) }{b(w,w)} \geq \lambda_j, 
\]
 which is the lower bound in \eqref{resEVmain}.
\end{proof}

A essential assumption for obtaining the (satisfactory) eigenvalue estimates in \eqref{resEVmain} is the penalization condition \eqref{assstab}. This condition requires that the penalization used in the $a_h(\cdot,\cdot)$ part of the discrete eigenproblem ``dominates'' the penalization used in the $b_h(\cdot,\cdot)$ part, cf. also the discussion in section~\ref{sectdiscussion}. 

\subsection{Analysis of a method with penalization and consistency errors} \label{sectgeneral}
In this section we generalize the analysis presented in the previous section in the sense that we consider a larger class of bilinear forms in    \eqref{defblfA} that in particular allows certain inconsistencies.
In the error analysis, besides the approximation quality quantity $\Theta_{h,j}$ defined above, we use \emph{consistency  parameters} introduced in the following assumptions.
\begin{assumption}[Consistency] \label{ass1} We assume that there are $\alpha_{h,i} <1$,  $\beta_{h,i} <1$, $i=1,2$, with $\alpha_{h,i} \downarrow 0,~\beta_{h,i}\downarrow 0$ for $h\downarrow 0$, such that the following holds for all $u,v \in  U_{\kmax}^e$:
\begin{align}
 \left| \tilde a_h( u, v) - a(\CEl u,\CEl v)\right| & \leq \alpha_{h,1} \Enorm{u}\Enorm{v},  \label{alpha1}\\
 | k_a( u, v)| & \leq \alpha_{h,2}\Enorm{u}\Enorm{v} ,  \label{alpha2}\\
 \left| \tilde b_h( u, v) - b(\CEl u,\CEl v)\right| & \leq \beta_{h,1} \|u\|_h\|v\|_h,  \label{beta1}\\
 | k_b( u, v)| & \leq \beta_{h,2} \|u\|_h \|v\|_h.  \label{beta2}
\end{align}
\end{assumption}

These conditions imply, with $\alpha_h:=\alpha_{h,1}+\alpha_{h,2}$, $\beta_h:=\beta_{h,1}+\beta_{h,2}$, the following inequalities for arbitrary $u,v \in  U_{\kmax}^e$:
\begin{align}
 \left| a_h( u, v) - a(\CEl u,\CEl v)\right| & \leq \alpha_h \Enorm{u}\Enorm{v},  \label{deltaa}\\
 \left| b_h ( u, v) - b(\CEl u, \cE_h^{-\ell}v) \right| & \leq \beta_h \|u^e\|_{h} \|v\|_h, \label{deltab} 
\end{align}
and $\alpha_h \downarrow 0$, $\beta_h \downarrow 0$ for $h\downarrow 0$. 
The consistency parameters $\alpha_h$, $\beta_h$ play a key role in the analysis below. We also need bounds for the consistency error if one of the two arguments $u,v$ is in the larger space $V_h+U_{\kmax}^e$. This is quantified in the following assumption.
\begin{assumption}[Consistency] \label{ass2}
 We assume that there are $\tilde \alpha_h<1$, $\tilde \beta_h<1$, with $\tilde \alpha_{h} \downarrow 0,~\tilde \beta_{h}\downarrow 0$ for $h\downarrow 0$, such that the following holds for all $u \in  U_{\kmax}^e$, $v \in V_h+U_{\kmax}^e$:
\begin{align}
 \left| a_h( u, v) - a(\CEl u,\CEl v)\right| & \leq \tilde \alpha_h\Enorm{u}\Enorm{v},  \label{deltaA}\\
 \left| b_h ( u, v) - b(\CEl u, \CEl v) \right| & \leq \tilde \beta_h\|u\|_{h} \|v\|_h. \label{deltaB} 
\end{align}
\end{assumption}
Note that in the simpified setting treated in section~\ref{sectpenal} we have $\tilde \alpha_h=\tilde \beta_h= \alpha_h=\beta_h=0$. Due to the fact that in Assumption~\ref{ass2} we allow $v \in V_h$, the consistency parameters $\tilde \alpha_h$ and $\tilde \beta_h$ can be significantly larger that $\alpha_h$ and $\beta_h$, respectively, cf. also section~\ref{sectTraceFEM}.  \emph{In the remainder we assume that Assumptions~\ref{ass1} and \ref{ass2} are satisfied}.
To simplify the presentation we assume, without loss of generality, that
\[
  \alpha_h<\tfrac12,~~\beta_h<\tfrac12,~~\tilde \alpha_h<\tfrac12,~~\tilde \beta_h<\tfrac12
\]
holds.  
 \begin{remark} \label{rembest}
 \rm  Consider the special case $H=\He$, in which we can avoid penaliziation, i.e. $k_a(\cdot,\cdot)=k_b(\cdot,\cdot)=0$. Furthermore assume $V_h=H$, i.e., there are no discretization errors.  Define  $a_h(\cdot,\cdot):=(1+\delta )a(\cdot,\cdot)$, $b_h(\cdot,\cdot):=(1-\delta)b(\cdot,\cdot)$ for $0 <\delta \ll 1$.  The eigenvalues of the inconsistent generalized eigenvalue problem with bilinear forms $a_h(\cdot,\cdot)$ and $b_h(\cdot,\cdot)$ are given by $\tilde \lambda_j=\lambda_j(1+\delta)(1-\delta)^{-1}=\lambda_j(1+2\delta +\mathcal{O}(\delta^2))$.  The (best possible) consistency parameters in \eqref{deltaa}-\eqref{deltab} are $\alpha_h=\beta_h=\frac{\delta}{1-\delta}=\delta +\mathcal{O}(\delta^2)$. This special case shows that (general) bounds for the eigenvalue error $|\lambda_j-\tilde \lambda_j|$ can not be better than of order $\mathcal{O}(\alpha_h+\beta_h)$.
\end{remark}

 Using \eqref{deltaa}-\eqref{deltab} we obtain a few easy corollaries that we will use in the analysis below.
\begin{corollary}\label{lemCourantF}
   Take $v \in U_{k_{\max}}^e$, $v \neq 0$. Then $a_h(v,v) >0$, $b_h(v,v)>0$ and  the following holds:
  \begin{align}
   \max \left\{ \left|1- \frac{a_h(v,v)}{a(\CEl v,\CEl v)} \right|\, , \, \left|1- \frac{a( \CEl v, \CEl v)}{a_h(v,v)}\right| \right\} & \leq \frac{\alpha_h}{1-\alpha_h} \leq 2 \alpha_h, \label{resCF1} \\
    \max \left\{ \left|1- \frac{b_h(v,v)}{b(\CEl v,\CEl v)} \right|\, , \, \left|1- \frac{b( \CEl v,\CEl v)}{b_h(v,v)}\right| \right\} & \leq \frac{\beta_h}{1-\beta_h} \leq 2 \beta_h. \label{resCF2}
  \end{align}
\end{corollary}
\begin{proof} Take $v \in  U_{\kmax}^e $, $v \neq 0$. The estimate \eqref{deltaa} takes the form
\[
 \left| a_h(v, v) - a(\CEl v,\CEl v)\right| \leq \alpha_ha_h(v, v).
\]
From this and $a(\CEl v,\CEl v) > 0$ it follows that $a_h(v,v)>0$ holds, and
using $\alpha_h\leq \tfrac12$ this  easily yields the result \eqref{resCF1}. The result \eqref{resCF2} can be derived analogously.
 \end{proof}
\ \\
\begin{remark} \label{remnorm} \rm The results above imply that $\Enorm{\cdot}$, $\|\cdot\|_h$ are norms on $U_{\kmax}^e$.
\end{remark}

\begin{corollary} \label{CorolE}
 For the extension operator $\cE$ and its left inverse the following bounds hold:
 \begin{align}
  \Enorm{\cE u}&  \leq (1+\alpha_h) \Ecnorm{u},~~\|\cE u\|_h \leq (1+\beta_h)\|u\|\quad \text{for all}~u \in U_{\kmax}, \label{boundE}\\
  \Ecnorm{\cE_h^{-\ell} v}&  \leq (1+\alpha_h) \Enorm{v},~~\|\cE_h^{-\ell} v\| \leq (1+\beta_h)\|v\|_h\quad \text{for all}~v \in U_{\kmax}^e. \label{boundEinv}
 \end{align}
 \end{corollary}
\begin{proof}
Take $v=\cE u =u^e$ in  \eqref{resCF1} and use $\cE_h^{-\ell}\cE u=u$.  We then get 
\[
  \left| 1- \frac{\Enorm{\cE u}^2}{\Ecnorm{u}^2}\right|\leq 2 \alpha_h.
\]
This yields $\Enorm{\cE u} \leq \sqrt{1+2\alpha_h}\Ecnorm{u}$, which implies the first estimate in \eqref{boundE}. The other results can be derived with the same arguments. 
 \end{proof}
\ \\
On $U_j^e$ the equivalence of the norms $\Enorm{\cdot}$ and $\|\cdot\|_h$ can be made explicit as follows. Using $\lambda_1 b(\CEl w,\CEl w) \leq a(\CEl w,\CEl w) \leq \lambda_j b(\CEl w,\CEl w)$ for all $w \in  U_j^e$ and the estimates in Corollary~\ref{CorolE} we obtain, for $1 \leq j \leq \kmax$,
\begin{equation} \label{normeq}
 \tfrac14 \lambda_1 \|w\|_h^2 \leq \Enorm{w}^2 \leq 4 \lambda_j \|w\|_h^2, \quad \text{for all}~w \in U_j^e.
\end{equation}
The lower inequality in \eqref{normeq} is a Friedrich's type estimate, which is an analogon of the estimate $\lambda_1 \|u\|^2 \leq \Ecnorm{u}^2 $ for all $u \in H$, which holds for the eigenvalue problem \eqref{eqevproblem}. We also need such an estimate  on the discretization space $V_h$. Such a  ``Friedrich's constant'' is introduced in the next assumption.
\begin{assumption}[Friedrich's inequality] \label{Friedrichs}
 We assume that there is $c_F >0$ independent of $h$ such that
 \begin{equation} \label{friedrichs}
    \|v\|_h \leq  c_F \Enorm{v} \quad \text{for all}~v \in V_h + U_{\kmax}^e.
 \end{equation}
\end{assumption}

 Due to the fact that we do not have the property \eqref{consprop},  for the projections $\Paj$ and $\Pbj$ defined as in \eqref{projectionahj}, we  do \emph{not} have a representation as in \eqref{proja}. Furthermore, the relation \eqref{fundrelation} does in general \emph{not} hold.
 In the analysis of in the previous section, for the case without  inconsistencies, we used \eqref{fundrelation} to derive the estimate  $\|\Paj w\|_h = \|\Pbj w\|_h \leq \|w\|_h$, i.e., the projection $\Paj $ has operator norm 1 also w.r.t. $\|\cdot\|_h$. We need a similar estimate for the case with inconsistencies considered in this section. This is derived in the following lemma.
 
\begin{lemma} \label{distP}
Take $v  \in V_h$, $v \neq 0$, and $1 \leq j \leq \kmax$. Assume that $\Enorm{v - \Paj v} \leq \tfrac12 \Enorm{v}$ holds.  With 
\begin{equation} \label{delv} \delta_{h,v}:=\frac{8}{\sqrt{3}} c_F \lambda_{\kmax}^\frac12(\tilde \alpha_h+ \tilde \beta_h)\frac{\Enorm{v - \Paj v}}{\Enorm{v}}
\end{equation}
the following estimates hold:
\begin{equation} \label{estPa1}
 (1-\delta_{h,v})\|P_{a_h,j}v\|_h  \leq  \|\Pbj v\|_h \leq (1+ \delta_{h,v})\|P_{a_h,j}v\|_h. 
 \end{equation}
\end{lemma}
\begin{proof}
It is convenient to introduce the following ``extended'' bilinear forms $a^e(\cdot,\cdot)$ and $b^e(\cdot,\cdot)$, cf. \eqref{alpha1} and \eqref{beta1}, that are defined on $ \He \times \He$:
\begin{equation} \label{defaebe} \begin{split}
  a^e(u, v)& :=a(\cE_h^{-\ell}u,\cE_h^{-\ell}v),\\
   b^e(u, v)& :=b(\cE_h^{-\ell}u,\cE_h^{-\ell}v).
\end{split}
\end{equation}
For these bilinear forms the extended eigenvectors $u_k^e$, $k=1,2,\ldots$ are orthogonal eigenvectors in the subspace ${\rm range}(\cE)   \subset \He$, with the same eigenvalues $\lambda_k$ as in \eqref{eqevproblem}: $a^e(u_k^e,u_\ell^e)= \lambda_k \delta_{k\ell}$, $b^e(u_k^e,u_\ell^e) = \delta_{k\ell}$,  $k,\ell \in \Bbb{N}$. We define corresponding orthogonal projections onto $U_j^e$:
\[
  P_{b^e,j} v := \sum_{k=1}^j b^e(v,u_k^e)u_k^e,~ P_{a^e,j} v := \sum_{k=1}^j \lambda_k^{-1} a^e(v,u_k^e)u_k^e,
\]
for which $
   P_{b^e,j}=P_{a^e,j}
$
holds.
 Take $v \in V_h$ and define $w:=P_{a_h,j}v-P_{a^e,j}v \in U_j^e$. Using \eqref{deltaA} we get
 \begin{align*}
  \Enorm{w}^2 &= a_h(P_{a_h,j}v-P_{a^e,j}v,w) = a_h(v-P_{a^e,j}v,w)\\
     & = a_h(v-P_{a^e,j}v,w)- a^e(v-P_{a^e,j}v,w) \\
     &= a_h(w,v-P_{a^e,j}v)- a(\CEl w,\CEl (v-P_{a^e,j}v)) \leq \tilde \alpha_h\Enorm{w} \Enorm{v-P_{a^e,j}v} \\
     & \leq \tilde \alpha_h\Enorm{w}\big( \Enorm{w} +\Enorm{v-P_{a_h,j}v}\big).
 \end{align*}
This yields 
\begin{equation}
\Enorm{P_{a_h,j}v-P_{a^e,j}v}  \leq  2  \tilde \alpha_h\Enorm{v-P_{a_h,j}v}. \label{estP1}
\end{equation}
With similar arguments we get
\begin{equation}\label{estP2}
 \|P_{b_h,j}v-P_{b^e,j}v\|_h \leq  2 \tilde \beta_h\|v-P_{b_h,j}v\|_h. 
\end{equation}
Note that
\begin{equation} \label{hhh}
 - \|P_{a_h,j} v-P_{b_h,j} v\|_h \leq  \|P_{a_h,j} v\|_h - \|P_{b_h,j} v\|_h \leq \|P_{a_h,j} v-P_{b_h,j} v\|_h.
\end{equation}
We derive a bound for $\|P_{a_h,j} v-P_{b_h,j} v\|_h$. For this we
define $\epsilon_{h,v}:=\frac{\Enorm{v - \Paj v}}{\Enorm{v}} \leq \tfrac12$. Due to orthogonality of the projection $\Paj$ we have $\Enorm{v} = (1-\epsilon_{h,v}^2)^{-\frac12} \Enorm{\Paj v} \leq \frac{2}{\sqrt{3}}\Enorm{\Paj v}$.
 Using this,  $P_{a^e,j}=P_{b^e,j}$, \eqref{estP1}, \eqref{estP2}, \eqref{friedrichs} and \eqref{normeq} we get
 \begin{equation}\label{eq2}\begin{split}
                 \|P_{a_h,j} v-P_{b_h,j} v\|_h 
                 & \leq \|P_{a_h,j} v-P_{a^e,j} v\|_h +\|P_{b_h,j} v-P_{b^e,j} v\|_h \\
                 &  \leq  c_F \Enorm{P_{a_h,j}v-P_{a^e,j}v} + 2 \tilde \beta_h\|v-P_{b_h,j}v\|_h \\
                 & \leq 2 c_F \tilde \alpha_h\Enorm{v-P_{a_h,j}v}+2 \tilde \beta_h\|v-P_{b_h,j}v\|_h \\
                 & \leq 2 c_F (\tilde \alpha_h+ \tilde \beta_h)\Enorm{v-P_{a_h,j}v}=2 c_F (\tilde \alpha_h+ \tilde \beta_h)\epsilon_{h,v} \Enorm{v} \\
                 & \leq \frac{4}{\sqrt{3}}c_F (\tilde \alpha_h+ \tilde \beta_h)\epsilon_{h,v} \Enorm{\Paj v}\\
                 & \leq \frac{8}{\sqrt{3}}c_F \lambda_{\kmax}^\frac12 (\tilde \alpha_h+ \tilde \beta_h)\epsilon_{h,v} \|\Paj v\|_h = \delta_{h,v}\|\Paj v\|_h. 
                  \end{split}
                  \end{equation}
    Using this in \eqref{hhh} yields the estimates in  \eqref{estPa1}.
\end{proof}

 We consider the discrete eigenproblem \eqref{discrev} and derive bounds for the approximations $\tilde \lambda_j \approx \lambda_j$, $j=1,\ldots, \kmax$.  In the proofs of the results below we combine the arguments used in the proof of Theorem~\ref{thmEV} with perturbation arguments, based on Corollary~\ref{lemCourantF} and Lemma~\ref{distP}. To simplify the presentation, we assume (without loss of generality) that the approximabilty parameter $\Theta_{h,j}$,  defined in \eqref{deftheta} satisfies $\Theta_{h,j} \leq \tfrac12$, $j=1,\ldots, \kmax$. 
\begin{theorem} \label{Thmeigenvalues1}
For $1 \leq j \leq \kmax $ the following holds:
\begin{align} 
 \lambda_j & \geq (1-2 \alpha_h)(1-2 \beta_h)(1- \Theta_{h,j}^2)\big(1- 2 \hat c (\tilde \alpha_h+\tilde \beta_h) \Theta_{h,j}\big) \tilde \lambda_j \label{resEVmainA}
\end{align}  with $\hat c:= \frac{8}{\sqrt{3}}c_F \lambda_{\kmax}^\frac12$.
\end{theorem}
\begin{proof} 
Take $1 \leq j \leq k_{\max}$. Define $P_h(U_j^e)= : W_j \subset V_h$. We have ${\rm dim}(W_j)=j$ and $\Paj:   W_j \to U_j^e$ is an isomorphism. Using this and Corollary~\ref{lemCourantF} we get
\begin{align}
  \lambda_j &= \max_{v \in U_j} \frac{a(v,v)}{b(v,v)} =\max_{v \in U_j^e} \frac{a(\CEl v,\CEl v)}{b(\CEl v,\CEl v)} \nonumber \\ &=\max_{v \in U_j^e} \frac{a(\CEl v,\CEl v)}{a_h( v, v)}\cdot \frac{b_h( v, v)}{b(\CEl v,\CEl v)}\cdot \frac{a_h(v,v)}{b_h(v,v)} \nonumber
  \\ & \geq (1-2 \alpha_h)(1-2 \beta_h)\max_{v \in U_j^e}\frac{a_h(v,v)}{b_h(v,v)} \nonumber \\
  & = (1-2 \alpha_h)(1-2 \beta_h) \max_{w \in W_j} \frac{a_h(\Paj w,\Paj w)}{b_h(\Paj w, \Paj w)} \nonumber \\ & =
   (1-2 \alpha_h)(1-2 \beta_h)\max_{w \in W_j} \frac{a_h( w, w)}{b_h( w,  w)} \cdot \frac{\Enorm{\Paj w}^2}{\Enorm{w}^2} \cdot \frac{\|w\|_h^2}{\|\Paj w\|_h^2}. \label{hulp0}
\end{align}
Elementary properties of orthogonal projections yield, for $w \in W_j$,  $\frac{\Enorm{\Paj w}^2}{\Enorm{w}^2} \geq 1- \Theta_{h,j}^2$ and $\Enorm{w- \Paj w} \leq \Theta_{h,j} \Enorm{w}$.
Using  Lemma~\ref{distP} we get
\[
  (1-\delta_{h,w})\|\Paj w\|_h \leq \|\Pbj w\|_h \leq \|w\|_h,
\]
hence, 
$
 \frac{\|w\|_h}{\|\Paj w\|_h} \geq 1- \hat c (\tilde \alpha_h+\tilde \beta_h) \Theta_{h,j}
$, with $\hat c:= \frac{8}{\sqrt{3}}c_F \lambda_{\kmax}^\frac12$. 
Using this in \eqref{hulp0} and with the Courant-Fischer eigenvalue characterization we obtain
\begin{align*}
   \lambda_j & \geq (1-2 \alpha_h)(1-2 \beta_h)(1- \Theta_{h,j}^2)\big(1- \hat c (\tilde \alpha_h+\tilde \beta_h) \Theta_{h,j}\big)^2  \max_{w \in W_j} \frac{a_h( w, w)}{b_h( w,  w)}  \\ &
   \geq (1-2 \alpha_h)(1-2 \beta_h)(1- \Theta_{h,j}^2)\big(1- \hat c (\tilde \alpha_h+\tilde \beta_h) \Theta_{h,j}\big)^2 \tilde \lambda_j,
\end{align*}
which yields the result in \eqref{resEVmainA}.
\end{proof}
\ \\
For the eigenvalue estimate in the other direction $ \lambda_j \leq c \tilde \lambda_j$ we present two different results. In the lemma below we introduce a consistency condition on a subspace of \emph{discrete} eigenvectors. As we will see in our applications, the corresponding consistency parameters ($\hat \alpha_h$ and $\hat \beta_h$ below) can be significantly larger then $\alpha_{h,1}$ and $\beta_{h,1}$ used in Assumption~\ref{ass1}, and the resulting eigenvalue error estimate is not optimal. The reason why we derive the result in Lemma~\ref{Thmeigenvalues2} is that we need a \emph{convergence} of discrete eigenvalues result (given in Corollary~\ref{corolconvergence}) in the eigenvector error analysis in section~\ref{secteigenvector}. In Theorem~\ref{Thmeigenvalues3} we derive a result that avoids the consistency parameters $\hat \alpha_h$, $\hat \beta_h$ and instead uses a quantity that measures an  error in the eigenvector approximation.  
\begin{lemma} \label{Thmeigenvalues2}
We assume
\begin{equation} \label{assstabA}
 2 \tilde \lambda_{k_{\max}}  k_b(v,v) \leq k_a(v,v) \quad \text{for all}~v \in \tilde U_{k_{\max}},
\end{equation}
and that there are $\hat \alpha_h<1$, $\hat \beta_h<1$, with $\hat \alpha_h \downarrow 0$, $\hat \beta_h \downarrow 0$ for $h \downarrow 0$,   such that
\begin{equation} \label{consdiscrete}
\begin{split}
 \left|  \tilde a_h(v,v) - a(\CEl v, \CEl v)\right| & \leq \hat \alpha_h \Enorm{v}^2,~~
 \text{for all}~v \in \tilde U_{k_{\max}},\\
 \left|\tilde b_h(v,v)-  b(\CEl v,\CEl v)\right| & \leq \hat \beta_h \|v\|_h^2,\quad \text{for all}~v \in \tilde U_{k_{\max}}.
 \end{split}
\end{equation}
 Assume $ \hat \beta_h < \tfrac12$ is satisfied. The following holds:
 \begin{equation} \label{upper1}
 \tilde \lambda_j \geq  \frac {1-2 \hat \beta_h}{1 + 2 \hat \alpha_h} \lambda_j, \quad 1 \leq j \leq k_{\max}.
 \end{equation}
\end{lemma}
\begin{proof} 
The proof is along the same lines as in the second part of the proof of Theorem~\ref{thmEV}.
Note that
 \[
   \tilde \lambda_j = \max_{v \in \tilde U_j} \frac{a_h(v,v)}{b_h(v,v)} = \max_{v \in \tilde U_j} \frac{\tilde a_h( v , v)+ k_a(v,v)}{\tilde b_h( v, v) + k_b(v,v)} .
 \]
 Using \eqref{assstabA}  we get for $v \in \tilde U_j$:
 \[
  2 \tilde \lambda_{\kmax} k_b(v,v) \leq k_a(v,v) \leq a_h(v,v)\leq\tilde \lambda_{\kmax} (\tilde b_h(v,v) + k_b(v,v)),
 \]
which implies $k_b(v,v) \leq \tilde b_h(v,v)$, hence, 
\begin{equation} \label{77}
 b_h(v,v) = \tilde b_h(v,v) + k_b(v,v) \leq 2 \tilde b_h(v,v).
\end{equation}
Thus we have $ \tilde b_h(v,v) >0$ for all $v \in \tilde U_j \setminus\{0\}$. Furthermore, this yields injectivity of $\CEl$ on $\tilde U_j$, as follows. Take $ v\in \tilde U_j$ with $\CEl v=0$. From \eqref{consdiscrete} and \eqref{77} we get $|\tilde b_h(v,v)| \leq \hat \beta_h b_h(v,v) \leq 2\hat \beta_h \tilde b_h(v,v)$. This yields $\tilde b_h(v,v)=0$ and thus, due to \eqref{77}, $b_h(v,v)=0$, hence, $v=0$. Thus $\CEl$ is injective. 
For arbitrary $v \in \tilde U_j$ we have $a_h(v,v) \leq \tilde \lambda_j b_h(v,v)$. This implies
\[
  \tilde a_h(v, v) + k_a(v,v) - \tilde \lambda_j k_b(v,v) \leq \tilde \lambda_j \tilde b_h( v, v), 
\]
which, due to \eqref{assstabA}, implies
$
 \tilde  a_h( v, v)\leq \tilde \lambda_j \tilde b_h( v, v)$.
Using this and \eqref{assstabA} yields, for $v \in \tilde U_j$ with $k_b(v,v)>0$:
$
  \frac{\tilde a_h( v, v)}{\tilde b_h( v, v)} \leq \tilde \lambda_j \leq \tilde \lambda_{k_{\max}} < \frac{k_a(v,v)}{k_b(v,v)},
$
and thus
$
  \frac{\tilde a_h( v, v) +k_a(v,v)}{\tilde b_h( v,v)+k_b(v,v)} \geq \frac{\tilde a_h( v,v) }{\tilde b_h( v, v)}  
$
holds for all $v \in \tilde U_j$, $v \neq 0$.
 Using this we  obtain
\begin{equation} \label{44} 
  \tilde \lambda_j = \max_{v \in \tilde U_j} \frac{a_h(v,v)}{b_h(v,v)}  \geq  \max_{v \in \tilde U_j} \frac{\tilde a_h(v,v) }{\tilde b_h(v,v)} .
\end{equation}
Note that for $v \in \tilde U_j$ we have, using \eqref{consdiscrete},
\[ \begin{split}
\tilde a_h(v,v)  & \geq a(\CEl v, \CEl v)- \hat \alpha_h a_h(v,v) \geq a(\CEl v, \CEl v)-\hat \alpha_h \tilde \lambda_j b_h(v,v) \\ & \geq a(\CEl v, \CEl v)- 2\hat \alpha_h \tilde \lambda_j \tilde b_h(v,v).
\end{split} \]
Using this in \eqref{44} yields
\begin{equation} \label{45}
  (1 + 2 \hat \alpha_h) \tilde \lambda_j \geq \max_{v \in \tilde U_j} \frac{ a(\CEl v,\CEl v) }{\tilde b_h(v,v)}.
\end{equation}
Also note 
\[
  \tilde b_h(v,v) \leq b(\CEl v , \CEl v)+ \hat \beta_h b_h(v,v) \leq b(\CEl v , \CEl v)+ 2 \hat  \beta_h \tilde b_h(v,v), 
\]
hence, $(1- 2 \hat \beta_h) \tilde b_h(v,v) \leq b(\CEl v , \CEl v)$. Using this in \eqref{45} and with the Courant-Fischer theorem we finally obtain
\[\begin{split}
   (1 + 2 \hat \alpha_h) \tilde \lambda_j & \geq (1- 2 \hat \beta_h)\max_{v \in \tilde U_j} \frac{ a(\CEl v,\CEl v) }{ b(\CEl v,\CEl v)} \\  & = (1-2 \hat \beta_h)\max_{w \in \CEl (\tilde U_j)}\frac{ a(w,w) }{b(w,w)} \geq  (1-2 \hat \beta_h)\lambda_j, 
 \end{split}\]
which completes the proof.
\end{proof}

\begin{corollary}[convergence of eigenvalues] \label{corolconvergence}
\rm Assume that the conditions \eqref{assstabA} and \eqref{consdiscrete} are fulfilled and that  $\Theta_{h,j} \downarrow 0$  for $h \downarrow 0$. For simplicity we make the assumption (which holds in our applications) $\tilde \alpha_h^2 \leq c \alpha_h$, $\tilde \beta_h^2 \leq c \beta_h$ for a suitable constant $c$. The results \eqref{resEVmainA} and \eqref{upper1} imply the error estimate
\begin{equation} \label{error1}
 \frac{|\lambda_j -\tilde \lambda_j|}{\lambda_j} \leq c \max \{ \hat \alpha_h, \hat \beta_h, \alpha_h, \beta_h, \Theta_{h,j}^2\} + \text{h.o. terms}  \quad \text{($h \downarrow 0$)}.
\end{equation}
Hence, for $j \leq \kmax$, we have \emph{convergence} of the discrete eigenvalue $\tilde \lambda_j \to \lambda_j$ for $h \downarrow 0$, with an upper bound for the rate of convergence determined by $\max \{ \hat \alpha_h, \hat \beta_h, \alpha_h, \beta_h, \Theta_{h,j}^2\}$.
\end{corollary}

We now derive another eigenvalue error estimate in which instead of the consistency parameters $\hat \alpha_h$, $\hat \beta_h$ introduced above we use $\alpha_h$, $\beta_h$ (cf. \eqref{deltaa}-\eqref{deltab}) and a quantity that measures (in $\Enorm{\cdot}$) the distance between the discrete invariant space $\tilde U_m$ and the corresponding continuous one $U_m^e$:
\begin{equation} \label{defPhi}
 \Phi_{h,m}:= \max_{w \in \tilde U_m} \frac{\Enorm{w - P_{a_h,m}w}}{\Enorm{w}}, \quad 1 \leq m \leq \kmax.
\end{equation}
\begin{theorem} \label{Thmeigenvalues3}
 Assume that for an $m$ with $1 \leq m \leq \kmax$ the condition
 \begin{equation} \label{condPhi}
  \max \{\Phi_{h,m}, c_F \tilde \lambda_{\kmax} \Phi_{h,m}\} \leq \tfrac12
 \end{equation}
is satisfied. The following holds for al $1 \leq j \leq m$:
\begin{equation} \label{estEVmain}
 \tilde \lambda_j \geq   (1-2 \alpha_h)(1-2 \beta_h)\big(1+\hat c (\tilde \alpha_h+\tilde \beta_h)\Phi_{h,m}\big)^{-2}(1-c_F^2 \tilde \lambda_{\kmax}\Phi_{h,m}^2)\lambda_j, 
\end{equation}
with $\hat c:=\frac{8}{\sqrt{3}}c_F \lambda_{\kmax}^\frac12$.
\end{theorem}
\begin{proof}
Take $1 \leq j \leq m$. Note that 
 \begin{equation} \label{eveq1} \begin{split}
  \tilde \lambda_j & = \max_{v \in \tilde U_{j}} \frac{a_h(v,v)}{b_h(v,v)} = \max_{v \in \tilde U_j} \frac{a_h(P_{a_h,m} v, P_{a_h,m}v)}{b_h(P_{a_h,m} v , P_{a_h,m}v)} \cdot \frac{ \|P_{a_h,m}v\|_h^2}{\|v\|_h^2}\cdot \frac{\Enorm{v}^2}{\Enorm{P_{a_h,m}v}^2} \\ & \geq 
    \max_{v \in \tilde U_j} \frac{a_h(P_{a_h,m} v, P_{a_h,m}v)}{b_h(P_{a_h,m} v , P_{a_h,m}v)} \cdot \frac{ \|P_{a_h,m}v\|_h^2}{\|v\|_h^2}.
\end{split} \end{equation}
From Lemma~\ref{distP} we obtain
\[
  \frac{ \|P_{a_h,m}v\|_h^2}{\|v\|_h^2}\geq \big(1+\hat c (\tilde \alpha_h+\tilde \beta_h)\Phi_{h,m}\big)^{-2} \frac{ \|P_{b_h,m}v\|_h^2}{\|v\|_h^2},\quad \hat c:=\frac{8}{\sqrt{3}}c_F \lambda_{\kmax}^\frac12.
\]
For the projection error in $P_{b_h,m} v$ we get, for $v \in \tilde U_j$,
\[ \begin{split}
 \|v-P_{b_h,m} v\|_h & \leq \|v-P_{a_h,m} v\|_h \leq c_F \Enorm{v-P_{a_h,m} v} \\ & \leq c_F \Phi_{h,m} \Enorm{v}  \leq c_F \tilde \lambda_{\kmax}^\frac12 \Phi_{h,m}\|v\|_h. 
\end{split} \]
Thus we get
\[ \frac{ \|P_{a_h,m}v\|_h^2}{\|v\|_h^2 }\geq \big(1+\hat c (\tilde \alpha_h+\tilde \beta_h)\Phi_{h,m}\big)^{-2}(1-c_F^2 \tilde \lambda_{\kmax}\Phi_{h,m}^2).
\]
Define $W_j:=P_{a_h,m}( \tilde U_j) \subset U_m^e$. Note that $\dim (W_j)=j$. Using the consistency estimates in Corollary~\ref{lemCourantF} we get
\[ \begin{split}
\max_{v \in \tilde U_j} \frac{a_h(P_{a_h,m} v, P_{a_h,m}v)}{b_h(P_{a_h,m} v , P_{a_h,m}v)}
 & = \max_{w \in W_j} \frac{a_h(w,w)}{b_h(w,w)}
 \\ & \geq (1-2 \alpha_h)(1-2 \beta_h)\max_{w \in \CEl (W_j)} \frac{a(w,w)}{b(w,w)} \\
& \geq (1-2 \alpha_h)(1-2 \beta_h) \lambda_j.
\end{split} \]
Combining these results yields the estimate in \eqref{estEVmain}. 
\end{proof}
\begin{corollary} \label{corolconvergence2} \rm
 Assume that $\Theta_{h,j} \downarrow 0$ and $\Phi_{h,m} \downarrow 0$ for $h \downarrow 0$. The results \eqref{resEVmainA} and \eqref{estEVmain} imply the error estimate
\begin{equation} \label{error2}
 \frac{|\lambda_j -\tilde \lambda_j|}{\lambda_j} \leq c \max \{\alpha_h, \beta_h, \Theta_{h,j}^2,\Phi_{h,m}^2 \} + \text{h.o. terms}  \quad \text{($h \downarrow 0$)}, 1 \leq j \leq m.
\end{equation}
Hence, we obtain a rate of convergence determined by  $\max \{ \alpha_h, \beta_h, \Theta_{h,j}^2,  \Phi_{h,m}^2 \}$. An important difference between the results \eqref{error1} and \eqref{error2} is that in the latter the consistency parameters $\hat \alpha_h$ and $\hat \beta_h$ do not occur.  Furthermore, note that in \eqref{error2} the approximability parameters $\Theta_{h,j}$ and $\Phi_{h,m}$ occur in \emph{squared} form, whereas the consistency parameters $\alpha_h$, $\beta_h$ occur \emph{linearly}. 
\end{corollary}

\section{Eigenvector error analysis of the discrete eigenproblem with penalization and  inconsistency} \label{secteigenvector}
In this section we present an analysis of the errors in the eigenvector approximations resulting from the discrete problem \eqref{discrev}.
Our analysisis is based on the theory presented in  \cite{Yserentant}. In that paper an error analysis of the Rayleigh-Ritz method \emph{without penalization or consistency errors} is presented that shows how the error in an eigenvector (and eigenvalue) approximation can be bounded in terms of its best approximation in the ansatz space, in the same spirit as the more general results in \cite{Knyazev}. We generalize the results of  \cite{Yserentant} in the sense that we allow penalization and consistency errors, i.e., we generalize the analysis of \cite{Yserentant} to the abstract setting presented in section~\ref{sectabstractdiscrete}.

Besides the projection w.r.t. the energy norm \eqref{Phdef} we also need the orthogonal projection on $V_h$ w.r.t.  $b_h(\cdot,\cdot)$,  denoted by $Q_h:\, \He \to V_h$, i.e.,   (cf. Fig.~\ref{figprojection}) 
\[
  b_h(w,v_h)=b_h(Q_hw,v_h) \quad \text{for all}~v_h \in V_h, 
\]
and is given by $Q_hw= \sum_{i=1}^n b_h(w,\tilde u_i) \tilde u_i \in V_h$. 

Take a \emph{fixed}  $k \leq \kmax$ and let $(\lambda_k,u_k)=:(\lambda, u)$ be the (exact) eigenpair that one is interested in. Note that $\lambda$ may be a multiple eigenvalue, in which case we have $\lambda_j=\lambda$ for certain $j \neq k$ and the eigenspace corresponding to $\lambda$ has dimension larger than one.  We assume a given (small) neighborhood $\Lambda$ of $\lambda$, i.e., $\lambda \in \Lambda$. Corresponding to this neighborhood we define the $b_h$-orthogonal projection onto ${\rm span}\{\, \tilde u_i~|~\tlam_i \in \Lambda\,\}$:
\begin{equation} \label{defQhg}
  Q_h^\Lambda:\, \He \to V_h,  \quad Q_h^\Lambda w = \sum_{\tlam_i \in \Lambda} b_h(w,\tilde u_i) \tilde u_i,
\end{equation}
and the linear mapping
\[
  R_h^\Lambda:\, \He \to V_h,  \quad R_h^\Lambda w = \sum_{\tlam_i \notin \Lambda} \frac{\tlam_i}{\tlam_i -\lambda}b_h(w,\tilde u_i) \tilde u_i.
\]
 Finally, for measuring nonconformity we introduce the natural defect quantity
\begin{equation} \label{defdincon}
 d_\lambda(w,v):= a_h(w,v)- \lambda b_h(w,v),\quad w, v \in \He.
\end{equation}
Note that in the conforming Ritz-Galerkin case, i.e., $\He=H$, $a_h(\cdot,\cdot)=a(\cdot,\cdot)$, $b_h(\cdot,\cdot)=b(\cdot,\cdot)$, we have $d_\lambda(u,v)=0$ for all $v \in H$.

We assume a linear operator $I_{V_h}: \, {\rm range}(\cE) \to V_h$. In the applications, this will be a (quasi-)interpolation operator.

We present a result which is a variant of Lemma 3.1 in \cite{Yserentant}.
\begin{lemma} \label{lemmaA}
 Let $u=u_k \in H$ be an eigenvector corresponding to an eigenvalue $\lambda=\lambda_k$. For $u^e=\cE u$ the following relation holds:
 \begin{equation} \label{iderror}
  u^e- Q_h^\Lambda u^e= R_h^\Lambda (u^e-P_hu^e) +(I-Q_h)(u^e-P_hu^e) + \sum_{\tlam_i \notin \Lambda} \frac{1}{\tlam_i- \lambda} d_\lambda (u^e,\tilde u_i) \tilde u_i.
 \end{equation}
\end{lemma}
\begin{proof}
Using the definitions we obtain
\begin{align*}
 \lambda b_h(u^e,\tilde u_i) &= a_h(u^e,\tilde u_i) -\big(a_h(u^e,\tilde u_i)-\lambda b_h(u^e,\tilde u_i)\big) \\
   &= a_h(u^e,\tilde u_i)- d_\lambda(u^e,\tilde u_i) = a_h(\tilde u_i, P_h u^e) -d_\lambda(u^e,\tilde u_i) \\
    &= \tlam_i b_h(\tilde u_i, P_h u^e)- d_\lambda(u^e,\tilde u_i) \\
    & =  \tlam_i b_h(P_h u^e -u^e, \tilde u_i) +\tlam_i b_h(u^e,\tilde u_i) -d_\lambda(u^e,\tilde u_i).
\end{align*}
This yields $(\lambda - \tlam_i) b_h(u^e,\tilde u_i)=\tlam_i b_h(P_h u^e -u^e, \tilde u_i)-d_\lambda(u^e,\tilde u_i)$, and thus for $\tlam_i \notin \Lambda$:
\[
 b_h(u^e,\tilde u_i)= \frac{\tlam_i}{\tlam_i-\lambda}b_h(u^e-P_h u^e, \tilde u_i)+ \frac{1}{\tlam_i-\lambda}d_\lambda(u^e,\tilde u_i). 
\]
This yields
\begin{equation} \label{h1}
 \begin{split}
  (Q_h-Q_h^\Lambda) u^e &= \sum_{\tlam_i \notin \Lambda} b_h(u^e,\tilde u_i)\tilde u_i \\ 
   &= \sum_{\tlam_i \notin \Lambda}\frac{\tlam_i}{\tlam_i-\lambda}b_h(u^e-P_h u^e, \tilde u_i)\tilde u_i + \sum_{\tlam_i \notin \Lambda}\frac{1}{\tlam_i-\lambda}d_\lambda(u^e,\tilde u_i) \tilde u_i \\
   & = R_h^\Lambda (u^e-P_hu^e) + \sum_{\tlam_i \notin \Lambda}\frac{1}{\tlam_i-\lambda}d_\lambda(u^e,\tilde u_i) \tilde u_i.
 \end{split}
\end{equation}
Note that
\begin{align*}
 u^e - Q_h^\Lambda u^e &= (Q_h-Q_h^\Lambda) u^e + (I-Q_h) u^e \\
  &= (Q_h-Q_h^\Lambda) u^e + (I-Q_h) (u^e -  P_hu^e), 
\end{align*}
and  combining this with the result \eqref{h1} completes the proof.
\end{proof}

Note that $Q_h^\Lambda u^e$ ist the $b_h$-orthogonal projection of $u^e$ on the subspace spanned by the approximate eigenvectors $\tilde u_i$ with $\tlam_i \in \Lambda$. Hence the expression of the right handside in \eqref{iderror} describes how well the eigenvector ``extension'' $u^e =\cE u$ can be approximated in this subspace. This expression contains the projections $P_h$, $Q_h$ and the term  $d_\lambda (u^e, \cdot)$, which is related to nonconformity. We now derive bounds for this expression in terms of projection errors and consistency errors.
As shown in e.g. \cite{Knyazev,BonitoEV,Yserentant}  the rate of convergence of the Rayleigh-Ritz method critically depends on whether the considered eigenvalue $\lambda$ is well separated from the other eigenvalues or is  part of a cluster of eigenvalues. To measure this, the quantity 
\begin{equation} \label{gapparameter}
  \gamma_\Lambda:= \max_{\tlam_i \notin \Lambda} \frac{\tlam_i}{|\tlam_i - \lambda|}
\end{equation}
is introduced and will be used in the bounds derived below. First we give a bound in the norm $\|\cdot\|_h$ and then an error bound in the energy norm $\Enorm{\cdot}$ is derived. We will need a dual norm on $V_h$. For $L \in V_h'$ we define
\[
  \|L\|_{V_h'}:= \max_{v_h \in V_h} \frac{L(v_h)}{\Enorm{v_h}},
\]
i.e., we consider duality w.r.t. the scalar product $a_h(\cdot,\cdot)$ on $V_h$. Using the fact that $\hat u_i=\tilde \lambda_i^{-\frac12} \tilde u_i$, $i=1,\ldots,n$, is an  $a_h$-orthonormal basis of $V_h$ we obtain $\|L\|_{V_h'}= \left( \sum_{i=1}^n L(\hat u_i)^2\right)^\frac12$. 
\begin{theorem} \label{thm1} For $(\lambda,u)$ as in Lemma~\ref{lemmaA} the following holds:
\begin{equation} \label{estthm1}
 \| u^e- Q_h^\Lambda u^e\|_h \leq \max \{1,\gamma_\Lambda\} \| u^e- P_h u^e\|_h + \gamma_\Lambda \tlam_1^{-\frac12}  \|d_\lambda(u^e,\cdot)\|_{V_h'}.
\end{equation}
\end{theorem}
\begin{proof}
 For $w \in \He$ we have
 \begin{align*}
  \|R_h^\Lambda w\|_h^2 &= \sum_{\tlam_i \notin \Lambda} \left(\frac{\tlam_i}{\tlam_i -\lambda}\right)^2 b_h(w,\tilde u_i)^2 = \sum_{\tlam_i \notin \Lambda} \left(\frac{\tlam_i}{\tlam_i -\lambda}\right)^2 b_h(Q_hw,\tilde u_i)^2 \\
   & \leq \gamma_\Lambda^2 \sum_{i=1}^n b_h(Q_h w,\tilde u_i)^2 = \gamma_\Lambda^2 \|Q_hw\|_h^2.
 \end{align*}
Combining this with orthogonality properties we obtain
\begin{align*}
  & \|R_h^\Lambda(u^e - P_h u^e) +(I-Q_h)(u^e - P_h u^e)\|_h \\ & = \left(\|R_h^\Lambda(u^e - P_h u^e)\|_h^2
  + \|(I-Q_h)(u^e - P_h u^e)\|_h^2 \right)^\frac12 \\
   & \leq \max \{1,\gamma_\Lambda\} \left(\|Q_h(u^e - P_h u^e)\|_h^2 + \|(I-Q_h)(u^e - P_h u^e)\|_h^2 \right)^\frac12 \\ & = \max \{1,\gamma_\Lambda\} \|u^e - P_hu^e\|_h.
\end{align*}
For the nonconformity term in \eqref{iderror} we obtain
\begin{align}
&  \|\sum_{\tlam_i \notin \Lambda} \frac{1}{\tlam_i- \lambda} d_\lambda (u^e,\tilde u_i) \tilde u_i\|_h^2 =
 \sum_{\tlam_i \notin \Lambda} \left(\frac{1}{\tlam_i- \lambda}\right)^2 d_\lambda(u^e,\tilde u_i)^2 \nonumber \\
  & =\sum_{\tlam_i \notin \Lambda} \left(\frac{\tlam_i}{\tlam_i- \lambda}\right)^2 \tlam_i^{-1} d_\lambda(u^e,\hat u_i)^2  \label{NCterm}\\
   & \leq \gamma_\Lambda^2 \tlam_1^{-1} \sum_{i=1}^n d_\lambda(u^e, \hat u_i)^2 = \gamma_\Lambda^2 \tlam_1^{-1} \|d_\lambda(u^e,\cdot)\|_{V_h'}^2.\label{NCterm1}
\end{align}
Combining these estimates completes the proof.
\end{proof}

\begin{remark} \label{remNonopt} \rm
 We comment on how the term $\gamma_\Lambda \tlam_1^{-\frac12} \|d_\lambda(u^e,\cdot)\|_{V_h'}$ that occurs in the bound \eqref{estthm1} can be improved. First note that from the estimate $\frac{\tlam_i}{|\tlam_i -\lambda|} \geq \frac{\tlam_1}{\tlam_1+\lambda}$, in which the lower bound for $\lambda=\lambda_k$, $k \leq \kmax$, is bounded away from zero, it follows that using  $\frac{\tlam_i}{|\tlam_i -\lambda|} \leq \gamma_\Lambda$ in \eqref{NCterm}-\eqref{NCterm1} is acceptable. This leads to the term $q^2:=\sum_{\tlam_i \notin \Lambda}  \tlam_i^{-1} d_\lambda(u^e,\hat u_i)^2$. In the estimate \eqref{NCterm1} we used $\tlam_i^{-1} \leq \tlam_1^{-1}$ and replaced $\sum_{\tlam_i \notin \Lambda}$ by the larger sum $\sum_{i=1}^n$. 
 We use the operator representation of the discrete surface Laplacian $L_h: V_h \to V_h$ defined by $a_h(u_h,v_h)=b_h(L_h u_h,v_h)$ for all $u_h,v_h \in V_h$. Hence $\Enorm{u_h}=\|L_h^\frac12 u_h\|_h$.
 In \eqref{NCterm1} we used the estimate
 \begin{equation} \label{st}
  q \leq \tlam_1^{-\frac12} \|d_\lambda (u^e, \cdot )\|_{V_h'}=  \tlam_1^{-\frac12} \max_{v_h \in V_h} \frac{d_\lambda (u^e, v_h )}{\|L_h^\frac12 v_h\|_h}.
 \end{equation}
This possibly too pessimistic estimate can be avoided as follows. We introduce the subspace $W_h = {\rm span}\{\tilde u_i~|~ \tlam_i \notin \Lambda\,\}$. Elementary arguments show that 
\begin{equation} \label{st1} 
q = \max_{v_h \in W_h}\frac{d_\lambda (u^e, v_h )}{\|L_h v_h\|_h}
\end{equation}
holds. Comparing this with \eqref{st} we observe that in \eqref{st1} we have the smaller space $W_h$ and the significantly stronger norm $\|L_h v_h\|_h$. In our analysis we use \eqref{st}, because we are not able to derive bounds for \eqref{st1} that are significantly better than the bounds for \eqref{st} derived in Lemma~\ref{boundd} below. These bounds lead to optimal eigenvector error bounds in the energy norm, but to suboptimal estimates in the norm $\|\cdot\|_h$, cf. the discussion after Corollary~\ref{corolmain}. 
\end{remark}

We now derive an error bound in the energy norm.
\begin{theorem} \label{thm2} For $(\lambda,u)$ as in Lemma~\ref{lemmaA} the following holds:
\begin{equation} \label{estthm2}
\begin{split}
 \Enorm{ u^e- Q_h^\Lambda u^e} & \leq (\gamma_\Lambda +1) \left( \tlam_n^{\frac12} \| u^e- I_{V_h} u^e\|_h + 3 \Enorm{u^e -I_{V_h}u^e} \right) \\ & + \gamma_\Lambda   \|d_\lambda(u^e,\cdot)\|_{V_h'}.
 \end{split}
\end{equation}
\end{theorem}
\begin{proof}
For $w \in \He$ we have
 \begin{align*}
  \Enorm{R_h^\Lambda w}^2 &= \Enorm{\sum_{\tlam_i \notin \Lambda} \frac{\tlam_i}{\tlam_i -\lambda} b_h(Q_h w, \tilde u_i) \tilde u_i }^2 
    = \Enorm{\sum_{\tlam_i \notin \Lambda} \frac{\tlam_i}{\tlam_i -\lambda} \tlam_i^{-1} a_h(Q_h w, \tilde u_i) \tilde u_i }^2 \\
    & = \Enorm{\sum_{\tlam_i \notin \Lambda} \frac{\tlam_i}{\tlam_i -\lambda} a_h(Q_h w, \hat u_i) \hat u_i }^2   
   = \sum_{\tlam_i \notin \Lambda} \left(\frac{\tlam_i}{\tlam_i -\lambda}\right)^2 a_h(Q_hw , \hat u_i)^2  \\
   & \leq \gamma_\Lambda^2 \sum_{i=1}^n a_h(Q_h w,\hat u_i)^2 = \gamma_\Lambda^2 \Enorm{Q_hw}^2.
 \end{align*}
 Note that for $v_h \in V_h$ we have $\Enorm{v_h}^2 =a_h(v_h,v_h) \leq \tlam_n b_h(v_h,v_h) = \tlam_n \|v_h\|_h^2$ and
 $\Enorm{u^e-P_h u^e} \leq \Enorm{u^e- I_{V_h}}$. Using this and the estimate above we get
 \begin{align}
  & \Enorm{R_h^\Lambda(u^e-P_hu^e) + (I-Q_h)(u^e-P_hu^e)} \nonumber \\ & \leq \Enorm{R_h^\Lambda(u^e-P_hu^e)} +\Enorm{(I-Q_h)(u^e-P_hu^e)} \nonumber\\
  & \leq (\gamma_\Lambda +1)\Enorm{Q_h(u^e -P_hu^e)} + \Enorm{u^e-P_hu^e} \nonumber\\
  & \leq (\gamma_\Lambda +1) \left(\Enorm{Q_h(u^e -I_{V_h}u^e)} + \Enorm{Q_h(I_{V_h}u^e -P_h u^e)}\right) + \Enorm{u^e-I_{V_h}u^e} \nonumber\\
  & \leq  (\gamma_\Lambda +1) \left(\tlam_n^\frac12 \|Q_h(u^e -I_{V_h}u^e)\|_h + \Enorm{I_{V_h}u^e -P_h u^e}\right) + \Enorm{u^e-I_{V_h}u^e} \nonumber \\
  & \leq  (\gamma_\Lambda +1) \left(\tlam_n^\frac12 \|u^e -I_{V_h}u^e\|_h + \Enorm{I_{V_h}u^e -P_h u^e}\right) + \Enorm{u^e-I_{V_h}u^e} \nonumber \\
  & \leq  (\gamma_\Lambda +1) \left(\tlam_n^\frac12 \|u^e -I_{V_h}u^e\|_h +3 \Enorm{u^e -I_{V_h} u^e}\right). \label{est7} 
 \end{align}
We now consider the nonconformity term in \eqref{iderror}:
\begin{align*}
 & \Enorm{\sum_{\tlam_i\notin \Lambda} \frac{1}{\tlam_i-\lambda}d_\lambda (u^e,\tilde u_i)\tilde u_i}  =
  \Enorm{\sum_{\tlam_i\notin \Lambda} \frac{\tlam_i}{\tlam_i-\lambda}d_\lambda (u^e,\hat u_i)\hat u_i}\\
  & = \left(\sum_{\tlam_i\notin \Lambda} \left(\frac{\tlam_i}{\tlam_i-\lambda}\right)^2 d_\lambda (u^e,\hat u_i)^2 \right)^\frac12 
 \leq \gamma_\Lambda \left( \sum_{i=1}^n d_\lambda (u^e,\hat u_i)^2 \right)^\frac12 = \gamma_\Lambda \|d_\lambda(u^e,\cdot)\|_{V_h'}.
\end{align*}
Combining this with the estimate \eqref{est7} completes the proof.
 \end{proof}

\begin{remark}\rm 
Even for the conforming case $d_\lambda(\cdot,\cdot)=0$, the bound in Theorem~\ref{thm2} differs from the one derived in \cite{Yserentant}. In that paper stability (i.e., uniform boundedness) of the $b_h$-orthogonal projection $Q_h$ in the energy norm is assumed. We prefer the formulation above with an (interpolation) operator $I_{V_h}$ and the factor $\tlam_n^\frac12$ in the error bound. In \cite{Yserentant} a suitable interpolation operator $I_{V_h}$ is used to show that the stability assumption is satisfied in a finite element setting.
\end{remark}

In both results in Theorem~\ref{thm1} and Theorem~\ref{thm2}, in the error bound for the eigenvector approximation  we have a subspace (i.e., $V_h$) \emph{approximation} part and a \emph{nonconformity} part. In both theorems,  the nonconformity is quantified by the same term $\|d_{\lambda}(u^e,\cdot)\|_{V_h'}$. The approximation error is determined by the term $\|u^e- P_h u^e\|_h$ (Theorem~\ref{thm1}) and by $\|u^e - I_{V_h} u^e\|_h$ and $\Enorm{u^e - I_{V_h} u^e}$ (Theorem~\ref{thm2}). In our applications, bounds for these approximation terms are derived from (surface) finite element error analysis, cf. section~\ref{sectTraceFEM}. The  ``constants'' used in the two theorems are very explicit and depend essentially only on the gap parameter $\gamma_{\Lambda}$ and the largest discrete eigenvalue $\tilde \lambda_n$. Concerning the latter we note the following.  In our finite element applications $\tlam_n^{\frac12}$ scales like  $\tlam_n^{\frac12} \sim h^{-1}$. The growth of the factor $\tlam_n^{\frac12}$ (for $ h \downarrow 0$) can be compensated by the higher order (interpolation) error $\|u^e - I_{V_h}u^e\|_h$ compared to $\Enorm{u^e - I_{V_h}u^e}$, cf. \eqref{estthm2}. 

The  gap parameter $\gamma_\Lambda$, which is the same  as in the literature \cite{Knyazev,BonitoEV,Yserentant}, plays an important role. An elaborate discussion of this parameter is given in \cite[Section 3.2]{Knyazev}, cf. also \cite[Remark 3.4]{BonitoEV}. 
We briefly address this gap parameter below. First we derive a bound for the nonconformity term. For this we 
 use the consistency conditions formulated in Assumption~\ref{ass2}.

\begin{lemma} \label{boundd}
 Let $u = u_k \in H$ be an eigenvector corresponding to $\lambda= \lambda_k$, $k \leq \kmax$. 
 The following holds, with $\tilde \alpha_h$, $\tilde \beta_h$ as in Assumption~\ref{ass2}:
 \begin{equation} \label{estdl}
  \left\| d_\lambda(u^e,\cdot)\right\|_{V_h'} \leq \sqrt{2 \lambda} \left( \tilde \alpha_h+  \lambda c_F^2 \tilde \beta_h\right).
 \end{equation}
\end{lemma}
\begin{proof}
 For the eigenpair $(u,\lambda)$ we have $a(u,v)=\lambda b(u,v)$ for all $v \in H$. For $u^e=\cE u$  and $v_h \in V_h$ we obtain, using Assumption~\ref{ass2} and the Friedrich's inequality \eqref{friedrichs}:
 \begin{align*}
   |d_\lambda(u^e, v_h)|&= |a_h(u^e,v_h)-\lambda b_h(u^e,v_h)| \\
    & \leq |a_h(u^e,v_h) -a(\CEl u^e, \cE_h^{-\ell}v_h) | + \lambda | b_h(u^e,v_h) - b(\CEl u^e,\cE_h^{-\ell}v_h))| \\
     & \leq \tilde \alpha_h \Enorm{u^e} \Enorm{v_h} + \lambda \tilde \beta_h \|u^e\|_{h} \|v_h\|_h \\
     & \leq  \big( \tilde \alpha_h + c_F^2 \lambda)\Enorm{u^e} \Enorm{v_h}.
 \end{align*}
 Using \eqref{resCF1} and $a(u,u)=\lambda$ we get $\Enorm{u^e} \leq \sqrt{1+2 \alpha_h} \Ecnorm{u} \leq \sqrt{2 \lambda}$. 
 Combining these results yields the estimate \eqref{estdl}.
\end{proof}
\ \\[1ex]
{\bf The gap parameter $\gamma_\Lambda$}\\
We briefly discuss this gap parameter. For this discussion it is essential that we have \emph{convergence of eigenvalues}, i.e. 
\begin{equation} \label{convEV}
\lim_{h \downarrow 0} \tilde \lambda_j = \lambda_j,\quad 1 \leq j \leq \kmax. 
\end{equation}
 From Corollary~\ref{corolconvergence} it follows that, under the assumptions formulated in that corollary, we indeed have this convergence of eigenvalues property. 

As a first example, consider the case  of a simple eigenvalue $\lambda=\lambda_k$ that is well separated from the other ones, say $\min_{i \neq k}|\lambda-\lambda_i| \geq \delta\lambda$, with $\delta>0$. Hence, $\delta$ is a measure for the separation between $\lambda$ and neighboring eigenvalues. One can take the neighborhood $\Lambda=[\lambda - \tfrac12 \delta\lambda,\lambda + \tfrac12 \delta \lambda]$ and due to the convergence of eigenvalues property \eqref{convEV} it follows that, for $h$ sufficiently small,  $\gamma_\Lambda=\max_{i \neq k}\frac{\tlam_i}{|\tlam_i-\lambda|} \leq
1 +\max_{i \neq k} \frac{\lambda}{|\tlam_i-\lambda|} \leq 1+2 \delta^{-1}$. In this situation $Q_h^\Lambda$ is a projection on the \emph{one}-dimensional subspace spanned by $\tilde u_k$. 

If the eigenvalue $\lambda$ is multiple or belongs to a cluster of very close eigenvalues one has to chose $\Lambda$ accordingly.
To illustrate this, we consider the case that the first $k$ eigenvalues form a cluster (some or all of these may  be multiple) that is well separated from $\lambda_{k+1}$, with separation parameter $\delta$:
\[
  0 \leq \lambda_1 \leq \ldots \leq \lambda_{k} < \lambda_{k+1}, \quad \delta:= \frac{\lambda_{k+1}- \lambda_k}{\lambda_k}.
\]
For approximation of the eigenspace ${\rm span}(u_j^e)$, $1 \leq j \leq k$, we choose the neighborhood $\Lambda:=[0,\lambda_k+ \tfrac12 \delta \lambda_k]$ of $\lambda= \lambda_j$. Due to the convergence of eigenvalues property it follows that, for $h$ sufficiently small,  $\gamma_\Lambda=\max_{i >k }\frac{\tlam_i}{|\tlam_i-\lambda_j|} \leq
1 + \frac{\lambda_j}{\tlam_{k+1}-\lambda_j} \leq 1+2 \delta^{-1}$. 
In this  case $Q_h^\Lambda$ is the $b_h$-orthogonal projection on the discrete invariant space $\tilde U_k={\rm span}\{\tilde u_1, \ldots, \tilde u_k\}$. The results in Theorems~\ref{thm1} and \ref{thm2} should be interpreted as errors in the approximation of $u_j^e$, $1 \leq j \leq k$, by an element from this $k$-dimensional  space spanned by discrete eigenvectors.  

We consider one further case that we need for the approximability parameter $\Phi_{h,m}$ defined in \eqref{defPhi}. We assume that $\lambda_m$ (for an $m \leq \kmax$) is well-separated from $\lambda_{m+1}$, with separation parameter $\delta$ defined as above. We do not make any assumptions concerning separations  between the eigenvalues $\lambda_1, \dots, \lambda_m$. Define $\Lambda:=[0,\lambda_m+ \tfrac12 \delta \lambda_m]$.  For $(u,\lambda)=(u_j,\lambda_j)$, $1 \leq j \leq m$,  the result of Theorem~\ref{thm2} and $\gamma_\Lambda \leq 1+2 \delta^{-1}$ yields:
\begin{align} \label{estk}
 & \Enorm{ u_j^e- Q_h^\Lambda u^e}  \leq \epsilon_j \\
 & \epsilon_j:=2(1+\delta^{-1}) \left( \tlam_n^{\frac12} \| u_j^e- I_{V_h} u_j^e\|_h + 3 \Enorm{u_j^e -I_{V_h}u_j^e} \right)  + (1+2 \delta^{-1})   \|d_{\lambda_j}(u_j^e,\cdot)\|_{V_h'}. \nonumber
\end{align}
 The projection $Q_h^\Lambda$ maps onto  the space $\tilde U_m= {\rm span}\{\tilde u_1, \ldots,\tilde u_m\}$. Hence, the result \eqref{estk} implies
\[
 {\rm dist}_{\Enorm{\cdot}} ( u_j^e, \tilde U_m) \leq \epsilon_j, \quad 1 \leq j \leq m.
\]
Linear combination and  $\sqrt{1+2 \alpha_h} \Enorm{\sum_{j=1}^m \xi_j u_j^e} \geq \Ecnorm{\sum_{j=1}^m \xi_j u_j}= \sqrt{\sum_{j=1}^m \xi_j^2 \lambda_j}$ yield
\[
 {\rm dist}_{\Enorm{\cdot}} (w, \tilde U_m) \leq \sqrt{2 \sum_{j=1}^m \lambda_j^{-1} \epsilon_j^2} \,\, \Enorm{w}=:E_m \Enorm{w} \quad \text{for all}~w \in U_m^e.  
\]
Now assume that $h$ is sufficiently small such that $E_m <1$. Since ${\rm dim}(\tilde U_m)= {\rm dim}(U_m^e)=m$ we obtain for the approximability parameter $\Phi_{h,m}$, which is a measure for the distance between the subspaces  $\tilde U_m$ and $U_m^e$:
\begin{equation} \label{defPhi2}
 \Phi_{h,m} \leq E_m.
\end{equation}
Note that $E_m$ essentially depends only $\epsilon_j$, $1 \leq j \leq m$, i.e. on approximabilty of the extended eigenvectors $u_j^e$ in $V_h$ and on the defect quantity $\|d_{\lambda_j}(u_j^e,\cdot)\|_{V_h'}$, $1 \leq j \leq m$. 

\subsection{Discussion of results} \label{sectdiscussion}
We discuss a few key points of our abstract error analysis.\\[1ex]
\emph{Penalization used in $a_h(\cdot,\cdot)$, $b_h(\cdot,\cdot)$}. The bilinear forms $a_h(\cdot,\cdot)$, $b_h(\cdot,\cdot)$ used in the discetization in the space $V_h$ \eqref{discrev} must be stable. A minimal condition is that both are positive definite on $V_h$. For these stability properties the penalty terms $k_a(\cdot,\cdot)$, $k_b(\cdot,\cdot)$ are essential. A further important property is the Friedrich's inequality \eqref{friedrichs}, which mimics the property $\|u\| \leq \lambda_1^{- \frac12} \Ecnorm{u}$ for all $u \in H$ on the continous level. An ``appropriate scaling'' of the penalty terms is essential. For stability, these penalty terms should be ``sufficiently large''. On the other hand, we need good approximability properties in these norms, e.g. small values for the parameter $\Theta_{h,j}$ is \eqref{deftheta} and optimal interpolation error estimates in $\|\cdot\|_h$ and $\Enorm{\cdot}$ in the estimate \eqref{estthm2}. One further aspect, related to stability is an appropriate \emph{relative} scaling of the penalty terms, expressed in the conditions \eqref{assstab}, \eqref{assstabA}. In the pure Galerkin setting we have the stabilitty property $\tilde \lambda_j \geq \lambda_j$. For a similar property in the conconforming case with penalization, cf. the lower estimate in \eqref{resEVmain}, we need  a relative scaling condition as in \eqref{assstab}. These different conditions related to penalization lead in our specific discretization methods to a scaling $h^{-2}$ and 1 for the penalization of the normal component on $\Gamma_h$ in $k_a(\cdot,\cdot)$ and $k_b(\cdot,\cdot)$, respectively, and, for the TraceFEM,  scalings $h^{-1}$ and $h$ (cf. \eqref{scalings}) for the normal derivative volume stabilization terms in $k_a(\cdot,\cdot)$ and $k_b(\cdot,\cdot)$, respectively.  
\\
\emph{Different types of consistency conditions.} In the assumptions~\ref{ass1}, \ref{ass2} and \eqref{consdiscrete} we introduced different consistency conditions. The weakest are those in Assumptions~\ref{ass1}. These involve \emph{only} elements from the space spanned by the (extended) eigenvectors of the continuous problem. In our applications these are smooth functions, and the smoothness property leads to ``higher order'' estimates for the parameters $\alpha_h$, $\beta_h$ in Assumption~\ref{ass1}. In Assumption~\ref{ass2} and \eqref{consdiscrete}, also elements from the discretization space $V_h$ are involved, leading to worse consistency bounds. We need Assumption~\ref{ass2} to derive (sharp)  bounds for the projection operators in Lemma~\ref{distP} an for the eigenvector defect quantity $\|d_\lambda(u^e,\cdot)\|_{V_h'}$ in Lemma~\ref{boundd}. In the eigenvalue error estimates, e.g. \eqref{resEVmainA}, these ``worse'' consistency parameters $\tilde \alpha_h$, $\tilde \beta_h$ are multiplied with the ``small'' approximability parameter $\Theta_{h,j}$, which then results in satisfactory error bounds. \\
\emph{Resulting eigenvalue error bounds}. The main error estimates for the eigenvalues, derived in Theorems~\ref{Thmeigenvalues1} and \ref{Thmeigenvalues3} are explicit in the sense that all constants and relevant parameters are specified. If we make  the simplifying assumption $\tilde \alpha_h^2 \leq c \alpha_h$, $\tilde \beta_h^2 \leq c \beta_h$, the key quantities that determine the error bound are the consistency parameters $\alpha_h$, $\beta_h$, and the approximability parameters $\Theta_{h,j}$, $\Phi_{h,m}$, cf. Corollary~\ref{corolconvergence2}. Note that the error bound depends linearly on  the consistency parameters but quadraticly on the approximability parameters. This linear dependene can not be improved, cf. Remark~\ref{rembest}. The parameter $\Theta_{h,j}$ depends on approximabilty of \emph{all} eigenvectors $u_i^e$, $1 \leq i \leq j$, in the finite element space $V_h$. This dependence on all eigenvector approximations is \emph{sub}optimal in the sense as discussed in \cite{Knyazev}. In the setting of our applications it is reasonable to assume that all eigenvectors corresponding to the smallest $j \leq \kmax$ eigenvalues have comparable approximability properties. Hence, this may justify the use of $\Theta_{h,j}$.  The approximability parameter $\Phi_{h,m}$ measures the distance between a continuous and corresponding discrete invariant space of dimension $m \geq j$. This quantity is avoided in the eigenvalue error estimate in Lemma~\ref{Thmeigenvalues2}. The result in that lemma, however, involves the relatively large consistency parameters $\hat \alpha_h$, $\hat \beta_h$, which leads to a suboptimal error bound, cf. Corollary~\ref{corolconvergence}. \\
\emph{Resulting eigenvector error bounds}. The main error estimates for the eigenvector approximations, derived in Theorems~\ref{thm1} and \ref{thm2} are explicit in the sense that all constants and relevant parameters are specified. An important ``constant'' is the gap parameter $\gamma_\Lambda$. Apart from this gap paramater these error bounds are determined by natural interpolation or projection errors and the defect quantity $\|d_\lambda(u^e,\cdot)\|_{V_h'}$. The latter can be bounded in terms of the consistency parameters $\tilde \alpha_h$, $\tilde \beta_h$.

\section{Application to finite element discretizations of the surface Laplace eigenproblem} \label{sectTraceFEM}
We show how  the variational eigenvalue problem \eqref{varev} and its finite element discretizations \eqref{P2hh} (SFEM) and \eqref{P2h} (TraceFEM) can be analyzed in the general abstract analysis presented above. For both discretizations the choice of the spaces $H$, $\hat H$, $\He$, the bilinear forms $a(\cdot,\cdot)$, $b(\cdot,\cdot)$,  the extension operator $\cE$ and the lifting operator $\CEl$ are explained in Remark~\ref{remexample1}. For the eigenproblem discretization \eqref{discrev} the  bilinear forms
\[
  a_h(\bu,\bv)= \tilde a_h(\bu,\bv)+ k_a(\bu,\bv), \quad b_h(\bu,\bv)= \tilde b_h(\bu,\bv) + k_b(\bu,\bv),
\]
are as follows , cf. section~\ref{sectFEM}, with $E_h(\bu)  := \frac12 \big(\gradGh \bu + \gradGh \bu^T\big)$, $\quad E_{T,h}(\bu):=E_h(\bu) - u_N \bH_h$:
\begin{align*}
\tilde a_h(\bu,\bv) &:= \int_{\Gamma_h} \tr (E_{T,h}\big(\bu)^T E_{T,h}(\bv)\big)\, ds_h + \int_{\Gamma_h}\bP_h \bu_h \cdot \bP_h \bv_h \, ds_h, \\
 k_a(\bu,\bv)&:= h^{-2} \int_{\Gamma_h} (\bu \cdot \tilde{\bn}_h) (\bv \cdot \tilde{\bn}_h)  \, ds_h \quad \text{(SFEM),} \\
 k_a(\bu,\bv)&:= h^{-2} \int_{\Gamma_h} (\bu \cdot \tilde{\bn}_h) (\bv \cdot \tilde{\bn}_h)  \, ds_h + 
 h^{-1} \int_{\Omega_{\Theta}^{\Gamma}} (\nabla \bu \bn_h) \cdot (\nabla \bv \bn_h)  \, dx ~ \text{(TraceFEM),}
\\
 \tilde b_h(\bu,\bv) &= \int_{\Gamma_h} \bP_h \bu \cdot \bP_h \bv \,ds_h, \\
 k_b(\bu,\bv) &= \int_{\Gamma_h} (\bu \cdot \bn_h) (\bv \cdot \bn_h)\,ds_h\quad \text{(SFEM),} \\
 k_b(\bu,\bv) &= \int_{\Gamma_h} (\bu \cdot \bn_h) (\bv \cdot \bn_h)\,ds_h + h\int_{\Omega_{\Theta}^{\Gamma}} (\nabla \bu \bn_h) \cdot (\nabla \bv \bn_h)  \, dx ~ \text{(TraceFEM).}
\end{align*}
In both cases (SFEM and TraceFEM) the bilinear forms $a_h(\cdot,\cdot)$, $b_h(\cdot,\cdot)$ are scalar products on the finite element space $V_h$.  Furthermore, there is a constant $c_F$, independent of $h$ such that the Friedrich's inequality \eqref{friedrichs} holds. These properties are easy to derive, cf. \cite{hansbo2016analysis} (for SFEM) and \cite{Jankuhn4} (for TraceFEM). Recall that $k_g$ denotes the degree of the finite element  polynomials used in the geometry approximation, cf. \eqref{kg1} and \eqref{kg2}. Concerning the \emph{consistency parameters} we have the following result.
\begin{lemma} \label{lemconsistres}
 For both methods (SFEM and TraceFEM) the following estimates hold for the consistency parameters defined in Assumptions~\ref{ass1}, \ref{ass2} and \eqref{consdiscrete}:
 \begin{align}
  \max \{ \tilde \alpha_h, \tilde \beta_h,\hat \alpha_h, \hat \beta_h\} &  \leq c h^{k_g}, \label{estA}\\
  \max \{ \alpha_h,  \beta_h  \} & \leq c h^{k_g+1}, \label{estB}
 \end{align}
 with a suitable constant $c$ independent of $h$.
\end{lemma}
\begin{proof}
 Proofs of these results are given in the papers \cite{hansbo2016analysis} (for SFEM) and \cite{Jankuhn4} (for TraceFEM). These proofs are rather long and technical. For comparing derivatives on $\Gamma_h$ and $\Gamma$, transformation rules are used, e.g. $\nabla_{\Gamma_h} w= \bB^T \nabla_\Gamma w^\ell$ for a scalar valued function on $\Gamma_h$ with lifting denoted by $w^\ell$, and with a matrix $\bB$ that satisfies $\|\bB- \bP \bP_h\|_{L^\infty(\Gamma_h)} \leq c h^{k_g +1}$. We sketch  how bounds for $\tilde \beta_h$, $\hat \beta_h$, $\beta_h$ can be derived to illustrate the improvement of the estimate in \eqref{estB} compared to \eqref{estA}. We restrict to the SFEM (TraceFEM can be treated very similar). The difference in surface measure $ds_h$ on  $\Gamma_h$ and $ds$ on $\Gamma$ is described by $ds = \mu_h ds_h$, with $\|1- \mu_h\|_{L^\infty(\Gamma_h)} \leq c h^{k_g+1}$. For $\bu,\bv \in H^1(\Gamma_h)^3$ we have, with $\bu^\ell,\, \bv^\ell \in H^1(\Gamma)^3$ the lifting to $\Gamma$, 
 \begin{align}
| \tilde b_h(\bu,\bv) - b(\CEl \bu,\CEl \bv)| & = \big| \int_{\Gamma_h} \bP_h \bu \cdot \bP_h \bv \, ds_h
-  \int_{\Gamma} \bP \bu^\ell  \cdot \bP \bv^\ell  \, ds\big| \nonumber \\
 & = \big| \int_{\Gamma_h} \bP_h \bu \cdot \bP_h \bv \, ds_h
-  \int_{\Gamma_h} \bP \bu  \cdot \bP \bv\mu_h \, ds_h\big| \nonumber \\
& = \big| \int_{\Gamma_h} (\bP_h - \mu_h \bP)\bu \cdot  \bv \, ds_h \big| \label{pp} \\
 & \leq \|\bP_h - \mu_h \bP\|_{L^\infty(\Gamma_h)} \|\bu\|_{L^2(\Gamma_h)} \|\bv\|_{L^2(\Gamma_h)} \nonumber \\ & \leq c h^{k_g} \|\bu\|_{L^2(\Gamma_h)} \|\bv\|_{L^2(\Gamma_h)}. \nonumber
 \end{align}
For the consistency parameter $\hat \beta_h$, cf. \eqref{consdiscrete}, this yields $\hat \beta_h \leq c h^{k_g}$, i.e., the estimate in \eqref{estA} for $\hat \beta_h$. For the penalty term $k_b(\bu,\bv)$, with $\bu \in \cE(\bV_T)$, hence, $\bu \cdot \bn =0$, and $ \bv \in H^1(\Gamma_h)^3$ we get
\begin{equation} \label{bp1}
  k_b(\bu,\bv)= \int_{\Gamma_h} \bu\cdot(\bn_h -  \bn) \bv \cdot \bn_h \, ds_h \leq \|\bn_h -  \bn\|_{L^\infty(\Gamma_h)} \|\bu\|_{L^2(\Gamma_h)} \|\bv\|_{L^2(\Gamma_h)}.
\end{equation}
Combining this with $\|\bn_h -  \bn\|_{L^\infty(\Gamma_h)} \leq c h^{k_g}$ and the result \eqref{pp} this yields the estimate $\tilde \beta_h \leq c h^{k_g}$ for the consistency parameter $\tilde \beta_h$ defined in Assumption~\ref{ass2}. This yields the result \eqref{estA} for the parameter $\tilde \beta_h$. We now consider the case, as in Assumption~\ref{ass1}, that both arguments are in the space of extended eigenvectors, i.e., $\bu,\bv  \in \cE(U_{\kmax})$. This implies $\bP \bu = \bu$, $\bP \bv= \bv$. In that case the term $\bP_h - \mu_h \bP$ in \eqref{pp} can be replaced by $\bP(\bP_h - \mu_h \bP)\bP= \bP(\bP_h - \mu_h \bI)\bP$, for which an (improved) estimate $\|\bP(\bP_h - \mu_h \bI)\bP\|_{L^\infty(\Gamma_h)} \leq ch^{k_g+1}$ holds. This yields $\beta_{h,1} \leq ch^{k_g+1}$ in \eqref{beta1}. For the penalty term we obtain an (improved) estimate
\begin{equation} \label{bp2}\begin{split}
   k_b(\bu,\bv)& = \int_{\Gamma_h} \bu\cdot(\bn_h -  \bn) \bv \cdot (\bn_h- \bn) \, ds_h \leq \|\bn_h -  \bn\|_{L^\infty(\Gamma_h)}^2 \|\bu\|_{L^2(\Gamma_h)} \|\bv\|_{L^2(\Gamma_h)} \\ & \leq c h^{2 k_g}\|\bu\|_{L^2(\Gamma_h)} \|\bv\|_{L^2(\Gamma_h)}.
\end{split} \end{equation}
This yields $\beta_{h,2} \leq  c h^{2 k_g}$ for the parameter in \eqref{beta2}. Hence, we get $\beta_h = \beta_{h,1}+\beta_{h,2} \leq c h^{k_g+1}$, which is the result \eqref{estB} for the parameter $\beta_h$. 

The results \eqref{estA} and \eqref{estB} for the $\alpha$-parameters are more difficult to derive, due to the transformation rules for derives that are needed. For proofs we refer to \cite{hansbo2016analysis} and \cite{Jankuhn4}. We give a few comments concerning the proofs given in these papers. Estimates for the consistency of the $\tilde a_h(\cdot,\cdot)$ bilinear form are given in Lemma 5.5 in \cite{hansbo2016analysis} and Lemma 5.15 in \cite{Jankuhn4}. In the upper bounds derived there, instead of the desired norm $\Enorm{\bu}$  a term of the form $\|\bu\|_{H^1(\Gamma_h)} + k_a(\bu,\bu)^\frac12$ occurs. Note that by definition $k_a(\bu,\bu)^\frac12 \leq\Enorm{\bu}$ holds.  Hence, it remains to bound $\|\bu\|_{H^1(\Gamma_h)}$ in terms of $\Enorm{\bu}$ for $\bu \in V_h+ U_{\kmax}^e$. For $\bu \in V_h$ such a result follows from a ``discrete Korn's inequality'', cf. Lemma 5.16 in \cite{Jankuhn4} and Lemma 5.7 in \cite{hansbo2016analysis}. For $\bu^e= \cE(\bu) \in U_{\kmax}^e$ we can use the Korn's inequality in $\bV_T$, cf. \cite{Jankuhn1}, and \eqref{resCF1}: $\|\bu^e\|_{H^1(\Gamma_h)} \leq c \|\bu\|_{H^1(\Gamma)} \leq c \, a(\bu,\bu)^\frac12 \leq c \, a_h(\bu^e,\bu^e)^\frac12 = c \Enorm{\bu^e}$. For $\bu, \bv \in U_{\kmax}^e$ an  improved consistency bound $\sim h^{k_g+1}$ for $\tilde \alpha_h$  (for SFEM) is derived in Lemma 5.5 in \cite{hansbo2016analysis}. For  this bound to hold, one needs $H^2$-regularity of the arguments $\bu, \bv \in U_{\kmax}^e$, and in the resulting consistency estimates  factors $\|\bu^\ell \|_{H^2(\Gamma)}$ and $\|\bv^\ell \|_{H^2(\Gamma)}$ occur. For the eigenfunctions of the vector Laplacian, on a sufficiently smooth surface $\Gamma$,  we have an $H^2$-regularity property that can be used to control $\|\bu^\ell \|_{H^2(\Gamma)}$ in terms of $\|\bu^\ell\|_{L^2(\Gamma)} \leq c \Enorm{\bu}$. Thus we get a bound $\alpha_{h,1} \leq c h^{k_g+1}$ for the consistency parameter in \eqref{alpha1}. For deriving a bound for the penalty term $k_a(\cdot,\cdot)$ one can proceed as in \eqref{bp1}-\eqref{bp2}. Note, however, that in $k_a(\bu,\bv)$ we have a scaling of $\int_{\Gamma_h} (\bu\cdot \tilde \bn_h) (\bv \cdot \tilde \bn_h) \, ds_h$   with $\eta \sim h^{-2}$. To obtain satisfactory (optimal) consistency error bounds one needs an improved normal $\tilde \bn_h$ such that $\|\tilde \bn_h - \bn\|_{L^\infty(\Gamma_h)} \leq c h^{k_g+1}$ holds.
\end{proof}

The results above imply that the key Assumptions~\ref{ass1}, \ref{ass2} and \ref{Friedrichs} needed in our general error analysis are satisfied. In the eigenvalue and eigenvector bounds that were derived, 
besides the consistency parameters treated in Lemma~\ref{lemconsistres} also certain \emph{approximability} parameters are used. We now study these parameters. In the eigenvalue error estimates the approximability parameters $\Theta_{h,j}$ and $\Phi_{h,m}$ occur. In Lemma~5.3 in \cite{hansbo2016analysis} and Lemma 5.1 in \cite{Jankuhn4} interpolation error estimates 
of the form
\begin{equation} \label{intest}
 \Enorm{\bu^e - I_{V_h} \bu^e} \leq c h^k \|\bu\|_{H^{k+1}(\Gamma)}, \quad \bu \in V_T \cap H^{k+1}(\Gamma)^3
\end{equation}
for the surface and trace finite element spaces $V_h$ are derived. These interpolation operators are also optimal w.r.t. the weaker $\|\cdot\|_h$-norm:
\begin{equation} \label{intestL2}
\|\bu^e - I_{V_h} \bu^e \|_h \leq c h^{k+1} \|\bu\|_{H^{k+1}(\Gamma)}, \quad \bu \in V_T \cap H^{k+1}(\Gamma)^3.
\end{equation}
We assume that the surface has sufficient smoothness such that for the range of $k$ values that we consider the vector-Laplace eigenfunctions have $H^{k+1}$ regularity. This implies that for a given eigenfunction $\bu $ the regularity quantity $\|\bu\|_{H^{k+1}(\Gamma)}/\Enorm{\bu^e}$ is finite. These results imply an estimate
\begin{equation} \label{b1}
  \Theta_{h,j} \leq c_j h^k,    \quad 1 \leq j \leq \kmax,
\end{equation}
with $c_j$ depending on the regularity quantity of the first $j$ eigenfunctions. For obtaining an estimate for $\Phi_{h,m}$ we use the result \eqref{defPhi2}. For the term $E_m$ we need (only) bounds for the quantities $\epsilon_j$ defined in \eqref{estk}. For this we can use the interpolation error bounds \eqref{intest}-\eqref{intestL2}, the estimate $\tilde \lambda_n \leq c h^{-2}$ for the largest eigenvalue of the discrete problem and the result in Lemma~\ref{boundd} for the defect quantity $\|d_\lambda(\bu^e,\cdot)\|_{V_h'}$. Thus we get
\begin{equation} \label{estI}
 \epsilon_j \leq c h^{k} + c( \tilde \alpha_h + \tilde \beta_h) \leq c_\delta (h^k +h^{k_g}),\quad 1 \leq j \leq m,
\end{equation}
with a constant $c_\delta$ that depends on the gap parameter $\delta=\delta_m=\frac{\lambda_{m+1}-\lambda_m}{\lambda_m}$. Thus we obtain the estimate
\begin{equation} \label{b2}
  \Phi_{h,m} \leq c_\delta (h^k +h^{k_g}) 
\end{equation}
(with a possibly different constant $c_\delta$). 
Using the bounds above for the consistency and approximability parameters we obtain the following main result.
\begin{corollary} \label{corolmain}
 For the SFEM and TraceFEM defined in section~\ref{sectFEM} the following eigenvalue error bound holds:
 \begin{equation} \label{R1}
  \frac{|\lambda_j - \tilde \lambda_j|}{\lambda_j} \leq c_1 h^{k_g+1} +c_2 h^{2k}, \quad 1 \leq j \leq m \leq \kmax.
 \end{equation}
Take $1 \leq j \leq \kmax$ and $\Lambda$ a small neighborhood on $\lambda_j$, cf. section~\ref{secteigenvector}. For the $b_h$-orthogonal projection $Q_h^\Lambda \bu_j^e$ of $\bu_j^e$ onto the discrete invariant space ${\rm span}\{\, \tilde \bu_i~|~ \tilde \lambda_i \in \Lambda\,\}$ the following holds:
\begin{equation} \label{R2}
 \Enorm{\bu_j^e - Q_h^\Lambda \bu_j^e} \leq c \gamma_\Lambda (h^k + h^{k_g}).
\end{equation}
with the gap parameter $\gamma_\Lambda$ as in \eqref{gapparameter}.
\end{corollary}
\begin{proof}
 We use Corollary~\ref{corolconvergence2} combined with the estimates in \eqref{estB} and \eqref{b1}-\eqref{b2}. This yields the result \eqref{R1}. The result \eqref{R2} follows from Theorem~\ref{thm2}, Lemma~\ref{boundd} and the estimates \eqref{estA}, \eqref{intest}, \eqref{intestL2} and $\tilde \lambda_n \leq c h^{-2}$. 
\end{proof}

We comment on the results in Corollary~\ref{corolmain}. From the eigenvalue error analysis for the conforming Galerkin case (no penalization and no geometry errors) it is well-known that the convergence order $2k$ in the eigenvalue error bound \eqref{R1} is optimal. Also the order $k_g+1$ related to the geometry error in \eqref{R1} is optimal in the following sense, cf. also Remark~\ref{rembest}. For the surface approximation $\Gamma_h \approx \Gamma$ used in the SFEM and TraceFEM we have the sharp estimate ${\rm dist} (\Gamma_h,\Gamma ) \leq c h^{k_g+1}$. As a specific example, consider a sphere with radius $r$,  $\Gamma=B(0;r)$. For this case the smallest three eigenvalues are $\lambda_1=\lambda_2=\lambda_3=1$, corresponding to the three dimensional space of Killing vector fields. The ``lowest frequency'' eigenvalue $\lambda_4$ scales linearly with the area of the sphere $\lambda_4\sim r^2$. Assume that the exact surface corresponds to $r=1$ and due to geometry approximation we have $\Gamma_h=B(0; 1-h^{k_g+1})$. If only this geometry error is considered, i.e. there are no approximation errors ($V_h=H^1(\Gamma_h)^3$), we have an error $|\lambda_4 - \tilde \lambda_4| \sim h^{k_g+1}$. Hence the eigenvalue error caused by geometry approximation can not be better than of order $h^{k_g+1}$. We note that in numerical experiments, cf. the results in \cite{BonitoEV} and in section~\ref{sectExp}, we typically observe a rate of convergence higher than $h^{k_g+1}$. This is probably due to the fact that in the geometry approximation there occur systematic cancellation effects. For example, a uniform shrinking (or expansion) of the geometry as in the sphere example $\Gamma= B(0;1) \approx B(0; 1-h^{k_g+1}) =\Gamma_h$, is not realistic. Instead it may happed that ${\rm dist}(\Gamma_h,\Gamma) \sim h^{k_g +1}$ but $|\int_{\Gamma} 1 \, ds - \int_{\Gamma_h} 1 \, ds_h | \sim  h^{k_g +2}$. Such cancellation effects are \emph{not} considered in our error analysis. In \cite{BonitoEV}, for the scalar Laplace-Beltrami eigenvalue problem, an analysis of superconvergence effects w.r.t. geometry errors for the SFEM is presented. The eigenvalue error bound \eqref{R1} is suboptimal in the sense that we need approximability of \emph{all} (extended) eigenvectors $\bu_1^e, \dots,\bu_j^e$, cf. the discussion in \cite{Knyazev} for the conforming Galerkin case. 
\\
The estimate \eqref{R2} for the energynorm of the eigenvector approximation is of optimal order. \\
Finally, we briefly comment of an eigenvector error bound in the weaker norm $\|\cdot\|_h$, based on Theorem~\ref{thm1}. We expect that for the first term in the bound in \eqref{estthm1} an estimate $\| \bu_j^e- P_h \bu_j^e\|_h \leq c h^{k+1}$ can be shown to hold. For the second term we obtain, based on Lemma~\ref{boundd} and the estimate \eqref{estA}, $\|d_{\lambda_j}(\bu_j^e,\cdot)\|_{V_h'} \leq c h^{k_g}$, leading to an error bound $ \sim c(h^{k+1} + h^{k_g})$. Experiments, cf. Section~\ref{sectExp}, indicate an (expected) error behaviour $~h^{k+1} + h^{k_g+1}$. Hence, the bound that we obtain is suboptimal. To improve this, we need a better bound for the defect term $d_{\lambda_j}$, cf. Remark~\ref{remNonopt}, in particular an estimate $q \le c h^{k_g+1}$, with $q$ in \eqref{st1}. So far, however, we were not able to derive such a result.

\section{Numerical experiments} \label{sectExp} 
We present results for the vector-Laplace eigenproblem \eqref{varev} on the unit sphere. In this case we have an eigenvalue $\lambda_1=\lambda_2=\lambda_3=1$ with multiplicity 3 corresponding to the three dimensional space of Killing vector fields that consists of rotations around each of the three axes in $\Bbb{R}^3$. The next eigenvalue is $\lambda_4=\lambda_5=\lambda_6=2$ with multiplicity 3. Formulas for the corresponding eigenvectors are not known to us. 

We discretize this problem with the TraceFEM \eqref{P2h}, implemented in the software Netgen/NGSolve with ngsxfem \cite{ngsolve,ngsxfem}. For the construction of the local triangulation $\cT_h^\Gamma$
we start with an unstructured tetrahedral Netgen-mesh with $h_{max} = 0.5$ and locally refine the mesh  using a marked-edge bisection method (refinement of tetrahedra that are intersected by the surface).  After discretization we obtain a discrete generalization eigenvalue problem. The smallest discrete eigenvalues $\tilde \lambda_i$, $1 \leq i \leq 6$, and corresponding eigenvectors are determined with algorithms available in Netgen/NGSolve and the SciPy system \cite{SciPy}.

First we present results for the errors in the discrete eigenvalues $\tilde \lambda_1$, $\tilde \lambda_4$, shown in Fig.~\ref{figres1}. Theory predicts a convergence order $\mathcal{O}(h^{k_g+1}+h^{2k})$. 
\begin{figure}[ht!]
\begin{minipage}{0.48\textwidth}
    \includegraphics[width=0.8\textwidth]{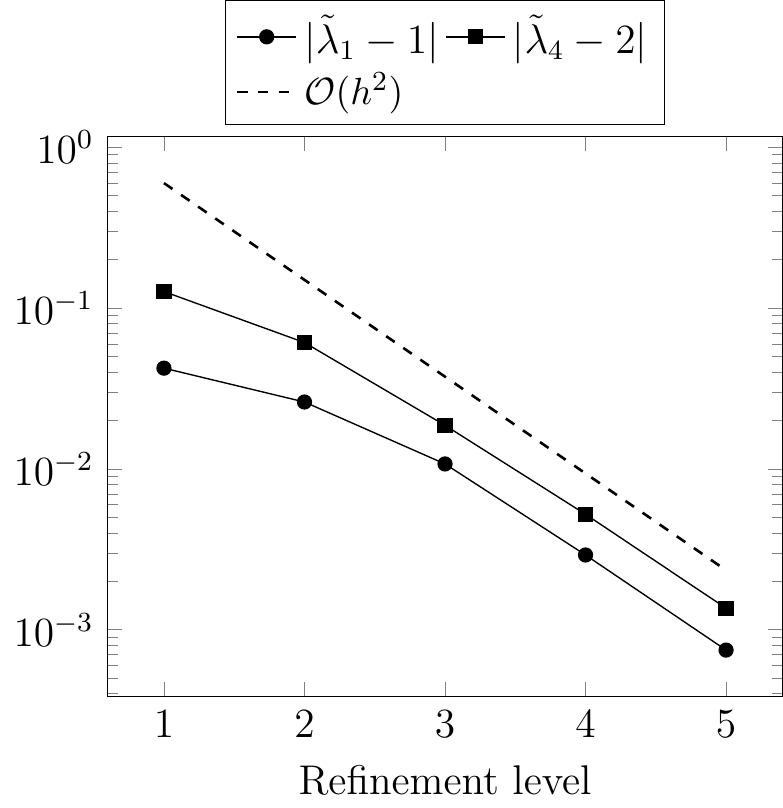}    
\end{minipage}
\begin{minipage}{0.48\textwidth}
       \includegraphics[width=0.8\textwidth]{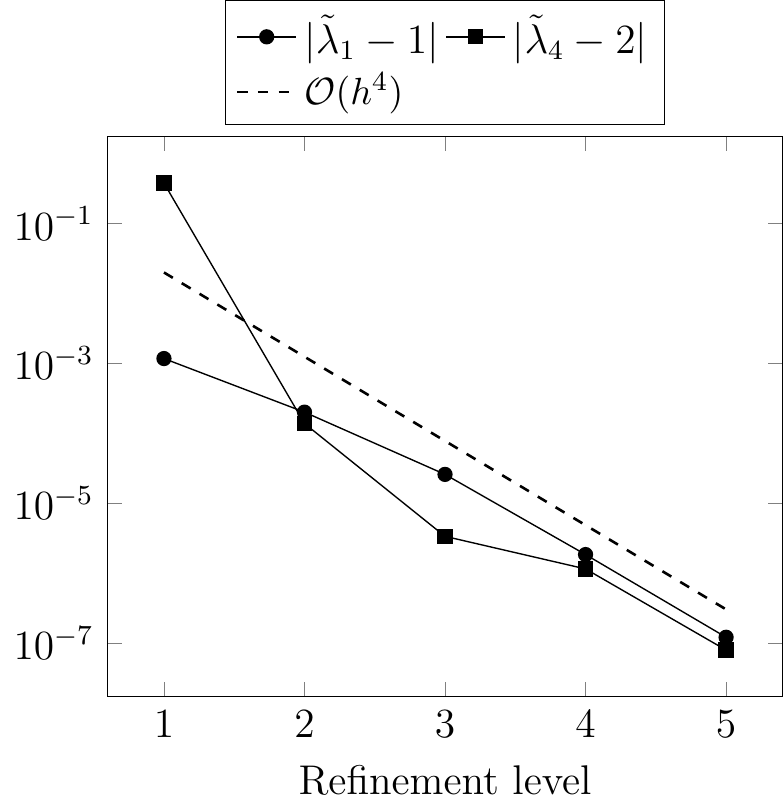}
\end{minipage}
\caption{Eigenvalue errors, $k=k_g=1$ (left) and $k=k_g=2$ (right)}\label{figres1}
\end{figure} 

We observe that for $k=k_g=2$ the convergence is faster as theory predicts. This might be related to a superconvergence that we observe for  the area aproximation $|\Gamma_h| \approx |\Gamma$, cf. Remark~\ref{remsuper}.
\begin{remark} \label{remsuper} \rm We briefly address the error between the exact surface area $|\Gamma|:= \int_\Gamma 1 \, ds$ and the area of the approximate surface $|\Gamma_h|:= \int_{\Gamma_h} 1 \, ds_h$.  First   note that the geometry error bound ${\rm dist}(\Gamma_h,\Gamma) \leq c h^{k_g+1}$ is sharp.  In the generic case one then has $\big| |\Gamma|-|\Gamma_h|\big| \leq c h^{k_g+1}$. Results for the surface area error, in our example of the TraceFEM for the unit sphere, are shown in Fig.~\ref{figres2} (left). We clearly observe a convergence order $h^4$ for the case $k_g=2$, which is one order better than the (expected) generic error bound $h^{k_g+1}=h^3$. 
\end{remark}

To investigate this further we consider the case $k=k_g=3$. Note that for $k_g=3$ we do not observe a superconvergence in Fig.~\ref{figres2} (left). The errors in the eigenvalues are shown in in Fig.~\ref{figres2} (right).  
\begin{figure}[ht!]
\begin{minipage}{0.48\textwidth}
    \includegraphics[width=0.8\textwidth]{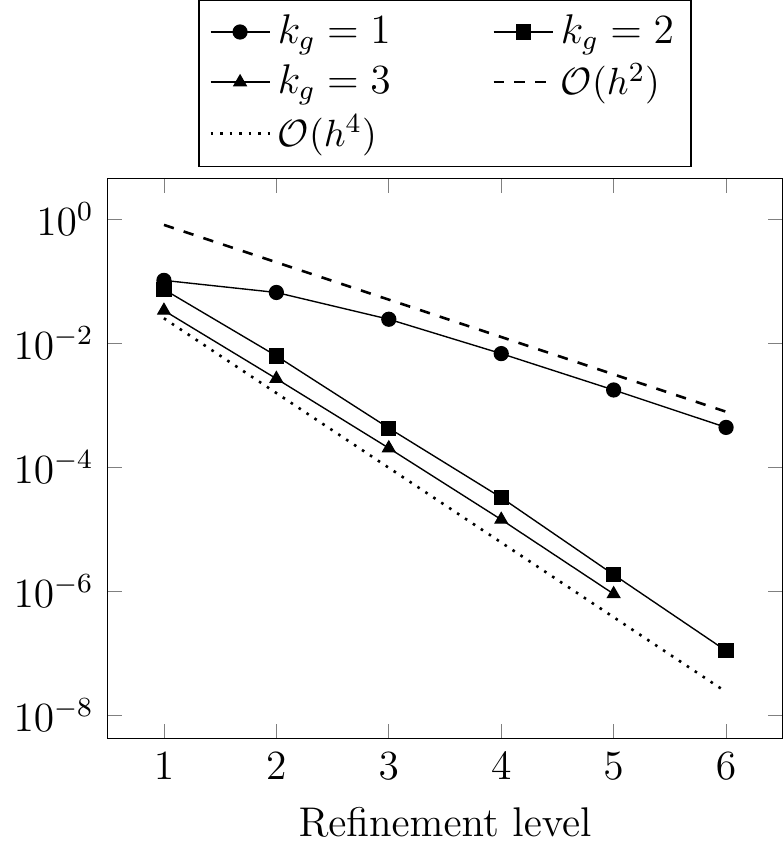}    
\end{minipage}
\begin{minipage}{0.48\textwidth}
       \includegraphics[width=0.8\textwidth]{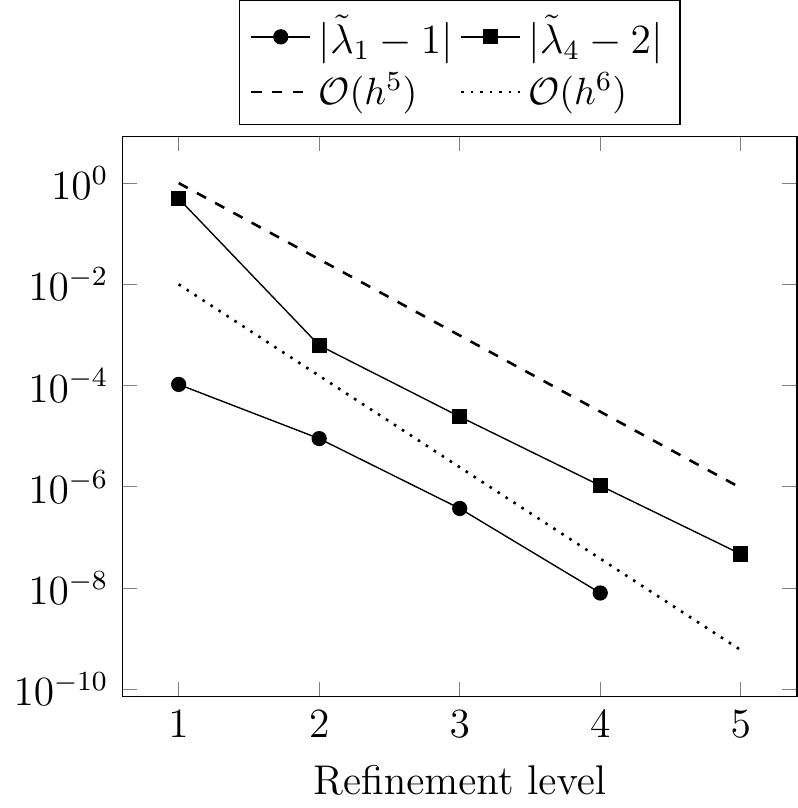}
\end{minipage}
\caption{Area error $\big| |\Gamma|-|\Gamma_h|\big|$ (left). Eigenvalue error for $k=k_g=3$ (right)}\label{figres2}
\end{figure}

We now observe for $\tilde \lambda_4$ a convergence that is (significantly) slower than $h^5$. The estimated convergence order between refinement levels $2$-$3$, $3$-$4$, $4$-$5$ is $4.68$, $4.53$ and $4.48$, respectively. This indicates that the convergence is dominated by the geometric error and its order is close to the theoretically predicted $h^{k_g+1}=h^4$.

We now consider  convergence of eigenvectors. For this we restrict to one of the Killing vector fields, namely $\bu_1(x,y,z)=(-y,x,0)$. We have a multiple eigenvalue $\lambda_1=\lambda_2=\lambda_3=1$ that is well-seperated from the rest of the spectrum. In the analysis in section~\ref{secteigenvector} we can choose a neighborhood $\Lambda=[0,1\tfrac12]$ and then obtain a gap parameter value $\gamma_\Lambda \sim 1$.  
 The $B_h(\cdot,\cdot)$ orthogonal projection on the space of discrete eigenvectors ${\rm span}\{\tilde\bu_1,\tilde \bu_2,\tilde \bu_3\}$, cf. \eqref{defQhg}, is given by
\[
  Q_h^\Lambda \bv = \sum_{j=1}^3 B_h(\bv,\tilde \bu_j) \tilde \bu_j.
\]
The error quantity that we consider is 
  $\Enorm{\bu_1^e -Q_h^\Lambda \bu_1^e}$. The exact eigenvector $\bu_1$ (and also $\bu_1^e$) is $C^\infty$ smooth, hence we have optimal approximation errors in the finite element space. The theory predicts an error bound, cf. Corollary~\ref{corolmain},  $\Enorm{\bu_1^e -Q_h^\Lambda \bu_1^e} \leq c (h^{k_g}+h^{k})$.
  Results for $\Enorm{\bu_1^e -Q_h^\Lambda \bu_1^e}$ and $\|\bu_1^e -Q_h^\Lambda \bu_1^e\|_{L^2(\Gamma_h)}$ are shown in the Figures~\ref{figres3} and \ref{figres4}.
\begin{figure}[ht!]
\begin{minipage}{0.48\textwidth}
    \includegraphics[width=0.8\textwidth]{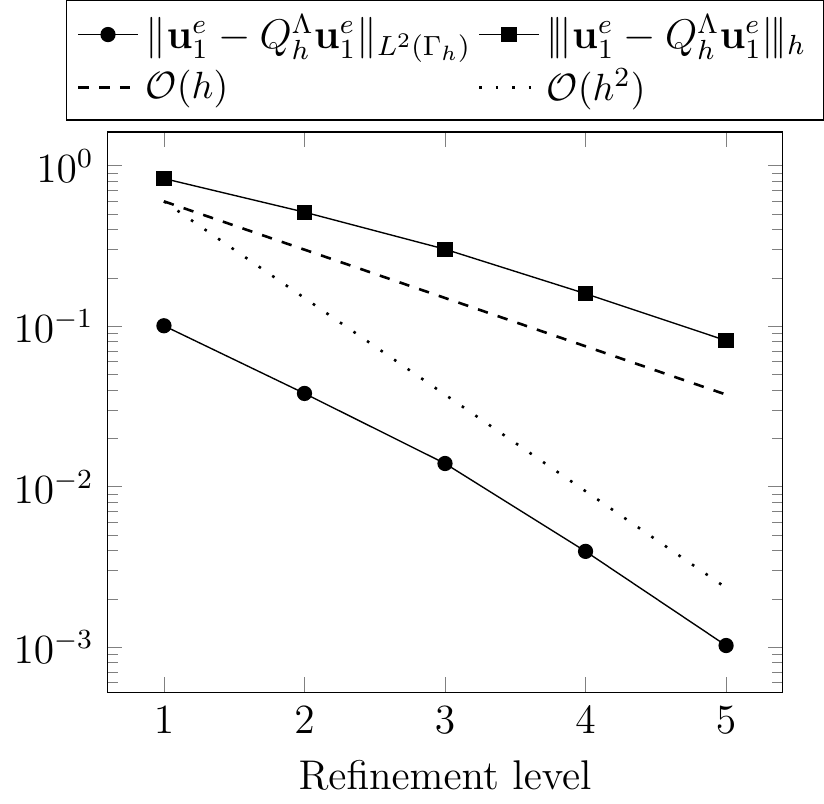}    
\end{minipage}
\begin{minipage}{0.48\textwidth}
       \includegraphics[width=0.8\textwidth]{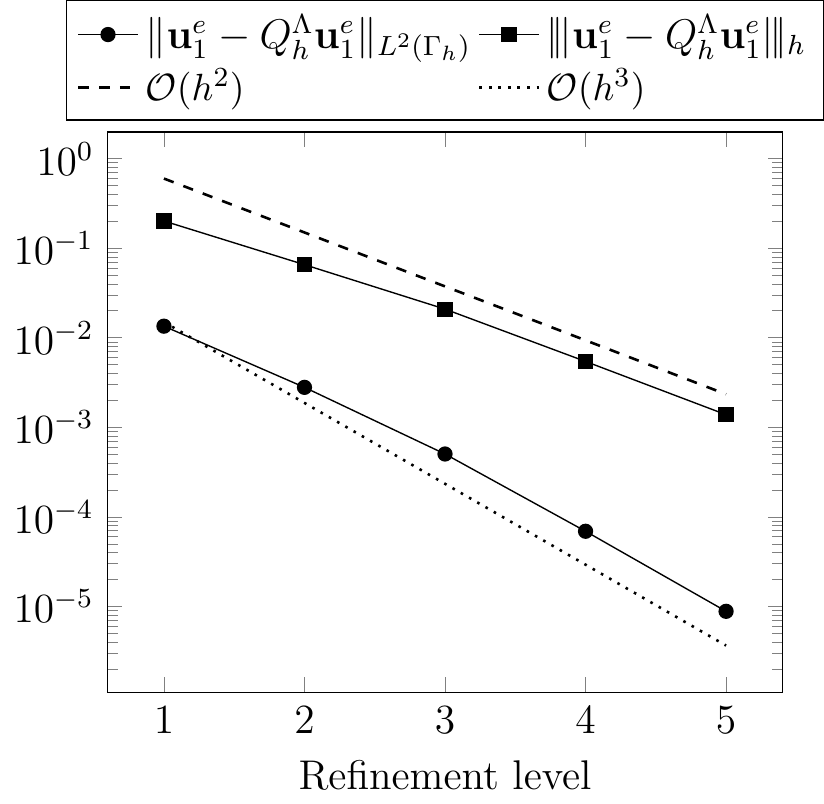}
\end{minipage}
\caption{Eigenvector errors, $k=k_g=1$ (left) and $k=k_g=2$ (right)}\label{figres3}
\end{figure}

In Fig.~\ref{figres3} one clearly observes the predicted $h^{k_g}+h^k$ convergence for the energy norm error. The error in the $L^2$-norm is one order better. To test, whether the term $h^{k_g}$ is sharp (in the scalar Laplace-Betrami case one has $h^{k_g+1}$), we take $k=k_g+1$. Results for $k_g=2$, $k_g=3$ are shown in Fig.~\ref{figres4}. These show that, for the energy norm error,  the bound  $h^{k_g}$ is sharp.
\begin{figure}[ht!]
\begin{minipage}{0.48\textwidth}
    \includegraphics[width=0.8\textwidth]{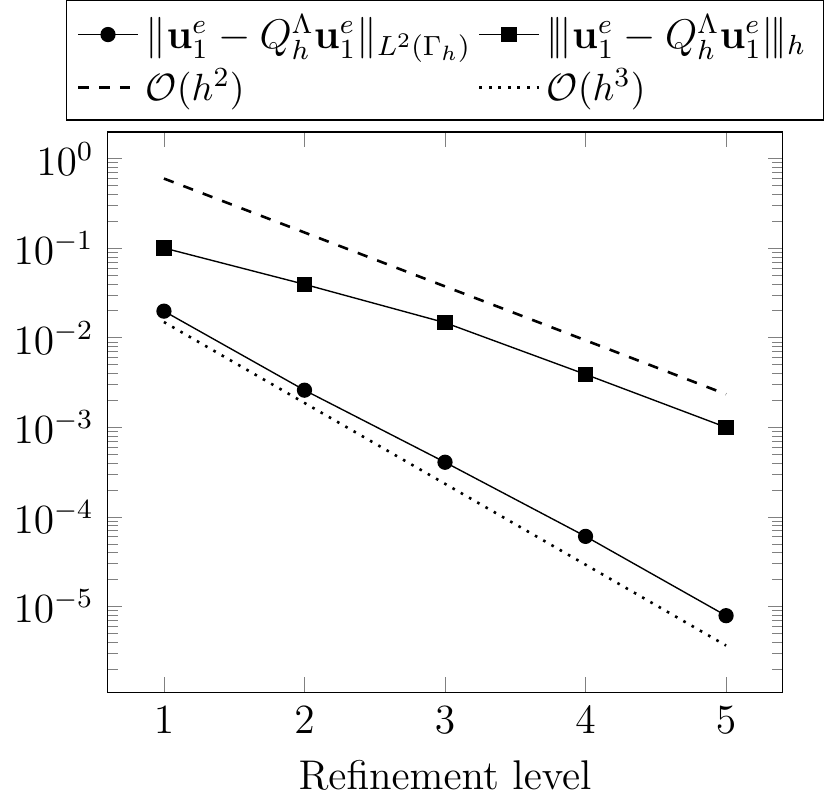}    
\end{minipage}
\begin{minipage}{0.48\textwidth}
       \includegraphics[width=0.8\textwidth]{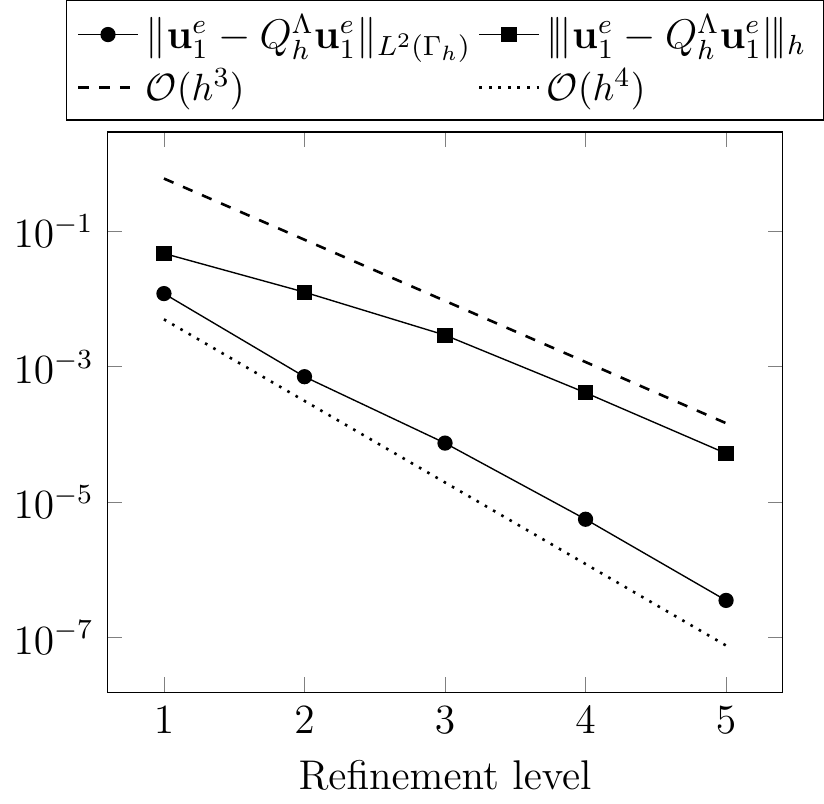}
\end{minipage}
\caption{Eigenvector errors, $k_g=2,k=3$ (left) and $k_g=3,k=4$ (right)}\label{figres4}
\end{figure}
\ 
\\[2ex]
{\bf Acknowledgment.} The author thanks Th.  Jankuhn for providing the results of the numerical experiments.  He also acknowledges financial support of the German Research Foundation (DFG)  within the ResearchUnit  “Vector-  and  tensor  valued  surface  PDEs”  (FOR  3013)  with  project  no.   RE1461/11-1.

\bibliographystyle{siam}
\bibliography{literatur}{}

\begin{thebibliography}{10}

\bibitem{ngsolve}
{\em {Netgen/NGSolve}}.
\newblock https://ngsolve.org/ (17 April 2019).

\bibitem{Antonietti}
{\sc P.~Antonietti, A.~Buffa, and I.~Perugia}, {\em Discontinuous {Galerkin
  approximation of the Laplace} eigenproblem}, Comput. Methods Appl. Mech.
  Engrg., 195 (2006), pp.~3483--3503.

\bibitem{arroyo2009}
{\sc M.~Arroyo and A.~DeSimone}, {\em Relaxation dynamics of fluid membranes},
  Phys. Rev. E, 79 (2009), p.~031915.

\bibitem{Azencot2013}
{\sc O.~Azencot, M.~Ben-Chen, F.~Chazal, and M.~Ovsjanikov}, {\em An operator
  approach to tangent vector field processing}, Computer Graphics Forum,
  (2013), pp.~73--–82.

\bibitem{Azencot2015}
{\sc O.~Azencot, M.~Ovsjanikov, F.~Chazal, and M.~Ben-Chen}, {\em {Discrete
  derivatives of vector fields on surfaces -- an operator approach}}, ACM
  Transactions on Graphics, 34 (2015), pp.~29:1--29:13.

\bibitem{Babuska}
{\sc I.~Babuska and J.~Osborn}, {\em Eigenvalue problems}, in Handbook of
  Numerical Analysis, Finite Element Methods (Part 1), Vol. II, P.~Ciarlet and
  J.~Lions, eds., North-Holland, 1991, Elsevier Science Publishers,
  pp.~641--787.

\bibitem{BenChen2010}
{\sc M.~Ben-Chen, A.~Butscher, J.~Solomon, and L.~Guibas}, {\em {On discrete
  Killing fields and patterns on surfaces}}, Eurographics Symposium on Geometry
  Processing, 29 (2010), pp.~1701--1711.

\bibitem{Boffi}
{\sc D.~Boffi}, {\em Finite element approximation of eigenvalue problems}, Acta
  Numerica, 19 (2010), pp.~1--120.

\bibitem{Bonito2019a}
{\sc A.~Bonito, A.~Demlow, and M.~Licht}, {\em A divergence-conforming finite
  element method for the surface {Stokes} equation}, arXiv:1908.11460,  (2019).

\bibitem{Bonito2019}
{\sc A.~Bonito, A.~Demlow, and R.~Nochetto}, {\em {Finite element methods for
  the Laplace-Beltrami operator}}, Handbook of Numerical Analysis, 21 (2020),
  pp.~1--103.

\bibitem{BonitoEV}
{\sc A.~Bonito, A.~Demlow, and J.~Owen}, {\em {A priori error estimates for
  finite element approximations to eigenvalues and eigenfunctions of the
  Laplace-Beltrami operator}}, SIAM J. Numer. Anal., 56 (2018), pp.~2693--2988.

\bibitem{Dyakonov}
{\sc E.~D'Yakonov}, {\em Optimization in Solving Elliptic Problems}, CRC Press,
  Boca Raton, 1996.

\bibitem{DEreview}
{\sc G.~Dziuk and C.~M. Elliott}, {\em Finite element methods for surface
  {PDEs}}, Acta Numerica, 22 (2013), pp.~289--396.

\bibitem{ebin1970groups}
{\sc D.~G. Ebin and J.~Marsden}, {\em Groups of diffeomorphisms and the motion
  of an incompressible fluid}, Annals of Mathematics, 92 (1970), pp.~102--163.

\bibitem{fries2018higher}
{\sc T.-P. Fries}, {\em Higher-order surface {FEM} for incompressible
  {Navier-Stokes} flows on manifolds}, International Journal for Numerical
  Methods in Fluids, 88 (2018), pp.~55--78.

\bibitem{grande2017higher}
{\sc J.~Grande, C.~Lehrenfeld, and A.~Reusken}, {\em Analysis of a high-order
  trace finite element method for pdes on level set surfaces}, SIAM J. Numer.
  Anal., 56 (2018), pp.~228--255.

\bibitem{grossvectorlaplace}
{\sc S.~Gross, T.~Jankuhn, M.~Olshanskii, and A.~Reusken}, {\em A trace finite
  element method for vector-{Laplacians} on surfaces}, SIAM J. Numer. Anal., 56
  (2018), pp.~2406--2429.

\bibitem{Jankuhn2}
{\sc S.~Gro{\ss}, T.~Jankuhn, M.~A. Olshanskii, and A.~Reusken}, {\em A trace
  finite element method for {Vector-Laplacians} on surfaces}, SIAM J. Numer.
  Anal., 56 (2018), pp.~2406--2429.

\bibitem{hansbo2016analysis}
{\sc P.~Hansbo, M.~G. Larson, and K.~Larsson}, {\em Analysis of finite element
  methods for vector {Laplacians} on surfaces}, IMA Journal of Numerical
  Analysis, 40 (2020), pp.~1652--1701.

\bibitem{Jankuhn1}
{\sc T.~Jankuhn, M.~A. Olshanskii, and A.~Reusken}, {\em Incompressible fluid
  problems on embedded surfaces: Modeling and variational formulations},
  Interfaces and Free Boundaries, 20 (2018), pp.~353--377.

\bibitem{Jankuhn4}
{\sc T.~Jankuhn and A.~Reusken}, {\em Trace finite element methods for surface
  {vector-Laplace} equations}, Journal of Numerical Mathematics,  (2019),
  p.~DOI: 10.1093/imanum/drz062.

\bibitem{Knyazev1}
{\sc A.~Knyazev}, {\em New estimates for {Ritz} vectors}, Math. Comp., 66
  (1997), pp.~985--995.

\bibitem{Knyazev}
{\sc A.~Knyazev and J.~Osborn}, {\em {New a-priori FEM error estimates for
  eigenvalues}}, SIAM J. Numer. Anal., 43 (2006), pp.~2647--2667.

\bibitem{Kobaetal_QAM_2017}
{\sc H.~Koba, C.~Liu, and Y.~Giga}, {\em {Energetic variational approaches for
  incompressible fluid systems on an evolving surface}}, Quart. Appl. Math., 75
  (2017), pp.~359--389.

\bibitem{Lederer2019}
{\sc P.~L. Lederer, C.~Lehrenfeld, and J.~Sch\"oberl}, {\em Divergence-free
  tangential finite element methods for incompressible flows on surfaces}, Int.
  J. Numerical Methods in Engineering, 121 (2020), pp.~2503--2533.

\bibitem{ngsxfem}
{\sc C.~Lehrenfeld}, {\em ngsxfem}.
\newblock https://github.com/ngsxfem (17 April 2019).

\bibitem{mitrea2001navier}
{\sc M.~Mitrea and M.~Taylor}, {\em {N}avier-{S}tokes equations on {Lipschitz}
  domains in {Riemannian} manifolds}, Mathematische Annalen, 321 (2001),
  pp.~955--987.

\bibitem{miura2017singular}
{\sc T.-H. Miura}, {\em On singular limit equations for incompressible fluids
  in moving thin domains}, Quart. Appl. Math., 76 (2018), pp.~215--251.

\bibitem{Nitschkeetal_arXiv_2018}
{\sc I.~Nitschke, S.~Reuther, and A.~Voigt}, {\em Hydrodynamic interactions in
  polar liquid crystals on evolving surfaces}, Phys. Rev. Fluids, 4 (2019),
  p.~044002.

\bibitem{nitschke2012finite}
{\sc I.~Nitschke, A.~Voigt, and J.~Wensch}, {\em A finite element approach to
  incompressible two-phase flow on manifolds}, Journal of Fluid Mechanics, 708
  (2012), pp.~418--438.

\bibitem{olshanskii2018finite}
{\sc M.~A. Olshanskii, A.~Quaini, A.~Reusken, and V.~Yushutin}, {\em A finite
  element method for the surface {Stokes} problem}, SIAM Journal on Scientific
  Computing, 40 (2018), pp.~A2492--A2518.

\bibitem{OlshanskiiZhiliakov2019}
{\sc M.~A. Olshanskii, A.~Reusken, and A.~Zhiliakov}, {\em Inf-sup stability of
  the trace {$P_2$-$P_1$ {Taylor-Hood} elements for surface {PDEs}}}, Math.
  Comp.,  (2019).

\bibitem{olshanskii2019penalty}
{\sc M.~A. Olshanskii and V.~Yushutin}, {\em A penalty finite element method
  for a fluid system posed on embedded surface}, Journal of Mathematical Fluid
  Mechanics, 21 (2019), p.~14.

\bibitem{Petersen}
{\sc P.~Petersen}, {\em Riemannian Geometry}, Springer, New York, 2016.

\bibitem{reusken2018stream}
{\sc A.~Reusken}, {\em Stream function formulation of surface {Stokes}
  equations}, IMA J. Numer. Anal., 20 (2020), pp.~109--139.

\bibitem{reuther2015interplay}
{\sc S.~Reuther and A.~Voigt}, {\em The interplay of curvature and vortices in
  flow on curved surfaces}, Multiscale Modeling \& Simulation, 13 (2015),
  pp.~632--643.

\bibitem{reuther2018solving}
\leavevmode\vrule height 2pt depth -1.6pt width 23pt, {\em Solving the
  incompressible surface {Navier-Stokes} equation by surface finite elements},
  Physics of Fluids, 30 (2018), p.~012107.

\bibitem{SciPy}
{\em {SciPy}}.
\newblock \verb|http://www.scipy.org|.

\bibitem{Solomon2011}
{\sc J.~Solomon, M.~Ben-Chen, A.~Butscher, and L.~Guibas}, {\em Discovery of
  intrinsic primitives on triangle meshes}, Computer Graphics Forum, 30 (2011),
  pp.~365–--374.

\bibitem{Tao2016}
{\sc M.~Tao, J.~Solomon, and A.~Butscher}, {\em {Near-isometric level set
  tracking}}, Eurographics Symposium on Geometry Processing, 35 (2016).

\bibitem{taylor1992analysis}
{\sc M.~E. Taylor}, {\em Analysis on {Morrey} spaces and applications to
  {Navier-Stokes} and other evolution equations}, Communications in Partial
  Differential Equations, 17 (1992), pp.~1407--1456.

\bibitem{Temam88}
{\sc R.~Temam}, {\em Infinite-dimensional dynamical systems in mechanics and
  physics}, Springer, New York, 1988.

\bibitem{Yserentant}
{\sc H.~Yserentant}, {\em {A short theory of the Rayleigh-Ritz method}},
  Computational Methods in Applied Mathematics, 13 (2013), pp.~495--502.

\end{thebibliography}

\end{document}